\documentclass[12pt]{amsart}

\usepackage{cancel}
\usepackage{latexsym, amssymb, amsmath,esint}
\usepackage{soul}
\usepackage{amsfonts, graphicx}
\usepackage{graphicx,color}
\newcommand{\be}{\begin{equation}}
\newcommand{\ee}{\end{equation}}
\newcommand{\beq}{\begin{eqnarray}}
\newcommand{\eeq}{\end{eqnarray}}

\usepackage{wasysym,stmaryrd}
\newtheorem{thm}{Theorem}[section]

\newtheorem{lma}[thm]{Lemma}
\newtheorem{prop}[thm]{Proposition}
\newtheorem{cor}[thm]{Corollary}
\newtheorem{defn}[thm]{Definition}

\theoremstyle{remark}
\newtheorem{rem}[thm]{Remark}
\numberwithin{equation}{section}

\newtheorem{claim}{Claim}[section]

\def\be{\begin{equation}}
\def\ee{\end{equation}}
\def\bee{\begin{equation*}}
\def\eee{\end{equation*}}

\def\vp{\varphi}

\def\K{K\"ahler }

\def\KR{K\"ahler-Ricci }
\def\Ric{\text{\rm Ric}}
\def\Rm{\text{\rm Rm}}

\def\ddb{\sqrt{-1}\partial\bar\partial}

\def\tr{\operatorname{tr}}

\def\de{\partial}
\def\e{\varepsilon}
\def\ve{\varepsilon}
\def\a{{\alpha}}
\def\b{{\beta}}

\begin{document}

\title[]
{On K\"ahler manifolds with non-negative mixed curvature}

\author{Jianchun Chu}
\address[Jianchun Chu]{School of Mathematical Sciences, Peking University, Yiheyuan Road 5, Beijing 100871, People's Republic of China}
\email{jianchunchu@math.pku.edu.cn}

\author{Man-Chun Lee}
\address[Man-Chun Lee]{Department of Mathematics, The Chinese University of Hong Kong, Shatin, Hong Kong, China}
\email{mclee@math.cuhk.edu.hk}

\author{Jintian Zhu}
\address[Jintian Zhu]{Institute for Theoretical Sciences, Westlake University, 600 Dunyu Road, 310030, Hangzhou, Zhejiang, People's Republic of China}
\email{zhujintian@westlake.edu.cn}

\renewcommand{\subjclassname}{\textup{2020} Mathematics Subject Classification}

\date{\today}

\begin{abstract}
In this work, we investigate compact K\"ahler manifolds with non-negative or quasi-positive mixed curvature coming from a linear combination of the Ricci and holomorphic sectional curvature, which covers various notions of curvature considered in the literature. Specifically, we prove a splitting theorem,  analogous to the Cheeger-Gromoll splitting theorem, for complete K\"ahler manifolds with non-negative mixed curvature containing a line, and then establish a structure theorem for compact K\"ahler manifolds with non-negative mixed curvature. We also show that the Hodge numbers of compact \K manifolds with quasi-positive mixed curvature must vanish. Both results are based on the conformal perturbation method.
\end{abstract}

\maketitle

\markboth{Jianchun Chu, Man-Chun Lee, Jintian Zhu}{On K\"ahler manifolds with non-negative mixed curvature}

\section{Introduction}\label{introduction}

In \K geometry, there are various notions of curvature, e.g. holomorphic bisectional curvature ($\mathrm{BK}$), Ricci curvature ($\Ric$), holomorphic sectional curvature  ($\mathrm{H}$), etc. K\"ahler manifolds with positive, quasi-positive (i.e. non-negative everywhere and strictly positive at some point) or non-negative curvature have been of great interest in complex geometry and algebraic geometry.

For instances, it was conjectured by Frankel \cite{Frankel1961} that a compact K\"ahler manifold with $\mathrm{BK}>0$ is biholomorphic to the complex projective space. The Frankel conjecture was resolved independently by Siu-Yau \cite{SiuYau1980} via differential geometry method and by Mori \cite{Mori1979} (he proved the more general Hartshorne conjecture, see \cite{Hartshorne1970}) via algebraic geometry method. As a natural continuation, Howard-Smyth-Wu \cite{HowardSmythWu1981} and Mok \cite{Mok1988} established the structure theorem for compact K\"ahler manifolds with $\mathrm{BK}\geq0$. That said, the universal cover of those compact \K manifolds split isometrically as product of Euclidean space with compact \K manifolds of certain type.  Since then, there has been interest in generalizing structural theory to \K manifolds with weaker curvature conditions. In particular, \K manifolds with $\Ric\geq 0$ and $\mathrm{H}\geq 0$ are very often to be considered, while $\Ric\geq 0$ and $\mathrm{H}\geq 0$ are independent to each other.

The Kodaira embedding theorem \cite{Kodaira1954} implies that any compact \K manifold with $\Ric>0$ is projective. Recall that a projective manifold is said to be rationally connected if any two points in it can be connected by a chain of rational curves. The work of Campana \cite{Campana1992} and Koll\'ar-Miyaoka-Mori \cite{KollarMiyaokaMori1992} showed that compact \K manifolds with $\Ric>0$ are rationally connected. In \cite{Yang2020}, it was further generalized by Yang to compact \K manifolds with quasi-positive $\Ric$, see also \cite{LeeLi2022} for an alternative approach. The structure theorem for compact K\"ahler manifolds with $\mathrm{Ric}\geq0$ was established by Campana-Demailly-Peternell \cite{CampanaDemaillyPeternell2015}.

On the other hand, the projectivity and rationally connectedness of compact \K manifolds with $\mathrm{H}>0$ was conjectured by Yau \cite[Problem 47]{Yau1982}. By studying the maximally rationally connected (MRC) fibration, Heier-Wong \cite{HeierWong2020} proved that projective \K manifolds with quasi-positive $\mathrm{H}$ are rationally connected. In \cite{Yang2018}, Yang proved that compact K\"ahler manifolds with $\mathrm{H}>0$ are all projective and thus confirmed the conjecture of Yau affirmatively. Very recently, Zhang-Zhang  \cite{ZhangZhang2023} generalized Yang's result to quasi-positive case, confirming a conjecture of Yang affirmatively. In \cite{Matsumura2022}, Matsumura established a structure theorem for projective manifolds with $\mathrm{H}\geq0$ by studying the structure of MRC fibration. Without projectivity assumption,  the structure theorem for \K manifolds with $\mathrm{H}\geq 0$ was given by Matsumura \cite{Matsumura2020} when $n=2$ and Zhang-Zhang  \cite{ZhangZhang2023} when $n=3$.

Motivated by the success in the study of $\Ric\geq 0$ and $\mathrm{H}\geq 0$, there has been interest in studying \K manifolds with non-negative intermediate curvatures between $\Ric$ and $\mathrm{H}$. For instances, 
Ni \cite{Ni2021a} initiated the notion of $k$-Ricci curvature ($\Ric_{k}$) for $1\leq k\leq n$ (see Subsection \ref{subsec:applications} for the definition) in order to study the $k$-hyperbolicity of K\"ahler manifolds. The $k$-Ricci curvature can be regarded as an intermediate curvature between the Ricci and holomorphic sectional curvature, i.e. $\Ric_{n}=\Ric$ and $\Ric_{1}=\mathrm{H}$. In \cite{Ni2021b}, Ni showed that a compact \K manifold with $\Ric_k>0$ is projective and rationally connected. It is natural to ask if any structure theorem still holds for \K manifolds with $\Ric_k\geq 0$.
Motivated by these, in this work we consider a more general curvature conditions called the mixed curvature, which was introduced by the first- second-named author and Tam \cite{ChuLeeTam2022}.  Precisely, the mixed curvature is defined as follows:
\begin{defn}
Let $(M^n,g)$ be a \K manifold. Define
\begin{equation*}
\mathcal{C}_{\a,\b,g}(X)=\a \Ric(X,\bar X)+\b  R(X,\bar X,X,\bar X)
\end{equation*}
for $X\in T^{1,0}M$ with $|X|_g=1$. We say that $\mathcal{C}_{\a,\b,g}\geq 0$ (resp. $\mathcal{C}_{\a,\b,g}>0$) if $\mathcal{C}_{\a,\b,g}(X)\geq 0$ (resp. $\mathcal{C}_{\a,\b,g}(X)>0$) for all $x\in M$ and unitary $X\in T^{1,0}_xM$. We will omit the index $g$ when the content is clear.
\end{defn}

This condition is natural in the sense that it is a linear combination of $\Ric\geq 0$ and $\mathrm{H}\geq0$. 
In contrast with the $k$-Ricci curvature, the mixed curvature is a weaker intermediate between the Ricci and holomorphic sectional curvature. Indeed, it was shown in \cite[Lemma 2.2]{ChuLeeTam2022} that by choosing $\a=k-1$ and $\b=n-k$, the sign of $\mathcal{C}_{k-1,n-k}$ are implied by that of $\Ric_k$. And more importantly, it behaves well analytically. In particular, the first-, second-named author and Tam \cite{ChuLeeTam2022} showed that \K manifolds with $\mathcal{C}_{\a,\b}<0$ where $\a,\b>0$ must be projective with ample canonical line bundle. This answers a question of Ni \cite{Ni2021a} in classifying compact \K manifolds with $\Ric_k<0$. As a natural counterpart, it is tempting to ask for a structure result related to $\mathcal{C}_{\a,\b}\geq 0$. Indeed, a recent work of Tang \cite{Tang2024} showed the projectivity and rationally connectedness of compact K\"ahler  manifolds with quasi-positive $\mathcal{C}_{\a,\b}$.

Our main result is the structure theorem for compact \K manifolds with $\mathcal{C}_{\a,\b}\geq 0$ where $\a>0$, $\b\neq 0$ and $\a+\b>0$. This covers compact \K manifolds with $\Ric_k\geq 0$ for $2\leq k\leq n-1$.

\begin{thm}\label{thm:structure}
Suppose $(M^n,g)$ is a compact \K manifold with $\mathcal{C}_{\a,\b}\geq0$ where $\a>0$, $\b\neq 0$ and $\a+\b>0$. Then the universal cover $(\tilde M,\tilde g)$ of $(M,g)$ splits isometrically as
\[
(\tilde M,\tilde g) \cong (\mathbb{C}^k,g_{\mathbb{C}^k}) \times \prod_{j} (X_j,h_j),
\]
where each $(X_j,h_j)$ is a compact, simply connected, projective \K manifold with $h^{2,0}(X_j)=0$. Moreover, each $X_j$ supports a \K metric (possibly different from $h_j$) with positive scalar curvature. In particular, the canonical line bundle $K_{X_{j}}$ is not pseudo-effective.
\end{thm}

The heart of the paper is Section \ref{sec:splitting}, where we prove a splitting theorem under $\mathcal{C}_{\a,\b}\geq0$. In contrast to Cheeger-Gromoll's splitting theorem \cite{CheegerGromoll1971}, due to $\beta\neq0$, the situation becomes more subtle since $\mathcal C_{\alpha,\beta}$ is not compatible with the Bochner formula. Our key observation is that the tangential Ricci curvature along geodesic lines vanishes, which is enough for  the Laplacian comparison type estimate along the line. Then we  adapt the idea from purely Riemannian geometry to find some foliation of $(M,g)$ by geodesic lines. This is motivated by the work of Liu \cite{Liu2013} and Chodosh-Eichmair-Moraru \cite{ChodoshEichmairMoraru2019} in constructing the foliation by area-minimizing objects.

Briefly speaking, we consider the conformal metric $\hat{g}_{\ve}=e^{-2\ve\vp}g$ such that $\mathcal{C}_{\a,\b,\hat{g}_{\ve}}>0$ in some annulus while $\hat{g}_{\ve}=g$ outside the outer ball. After a series of rather technical analysis, we show that the $\hat{g}_{\ve}$-geodesic segment must intersect the inner ball, and then the required $g$-geodesic line arises as a limit of $\ve\to0$. The main source of complication comes from the non-K\"ahlerity of $\hat{g}_{\ve}$. Instead of the conformal perturbation, one straightforward attempt is to consider the K\"ahler perturbation $g_{i\bar{j}}+\ve\vp_{i\bar{j}}$. However, since its curvature involves the fourth order derivative of $\vp$, it seems quite difficult to choose suitable $\vp$ satisfying the above curvature requirement. 

On the other hand, there are difficulties in carrying out the analogous argument of \cite{Matsumura2022} where Matsumura established the structure theorem for projective manifolds with $\mathrm{H}\geq0$ by studying the structure of MRC fibration. The first difficulty lies in constructing suitable MRC fibration (without indeterminacy locus) for non-project K\"ahler manifolds (see \cite[Theorem 1.5 and Section 5]{Matsumura2022}). Even if the background manifold is projective, there is another difficulty. One key step of \cite{Matsumura2022} is to show the tangent vector in the horizontal direction (of the MRC fibration) is truly flat, and then the tangent bundle admits the holomorphic orthogonal splitting. Since $\alpha\neq0$, our situation ($\mathcal{C}_{\a,\b}\geq0$) is quite different from $\mathrm{H}\geq0$ ($\mathcal{C}_{0,1}\geq0$). Roughly speaking, when $\alpha\neq0$, for the tangent vector $X$ in the horizontal direction, due to the existence of $\Ric(X,\bar X)$ in $\mathcal{C}_{\a,\b}(X)$, the analogous argument of \cite{Matsumura2022} seems not work here.

In \cite{Ni2021b}, Ni showed that a compact \K manifold with $\Ric_k>0$ has vanishing Hodge number $h^{p,0}$ for all $1\leq p\leq n$. By adapting the conformal perturbation method mentioned above, we are able to extend Ni's vanishing theorem to quasi-positive $\Ric_k$ curvature for $2\leq k\leq n$. More generally, we prove the following:
\begin{thm}\label{thm:quasi-positive-mixed}
Suppose $(M^n,g)$ is a compact \K manifold with quasi-positive $\mathcal{C}_{\a,\b}$ where $\a>0$ and $\a+\b>0$. Then the Hodge number $h^{p,0}(M)$ vanishes for all $1\leq p\leq n$. Moreover, $M$ is projective and simply connected.
\end{thm}

We note that the projectivity and rational connectedness (implies simply connectedness) were already proved by Tang \cite{Tang2024} under quasi-positive curvature condition, based on different method. 

\bigskip

The rest of the paper is organized as follows: In Section \ref{sec:mixed curvature}, we will generalize the notion of mixed curvature to Hermitian manifolds. In Section \ref{sec:geodesic line}, we will investigate the geodesic line under $\mathcal{C}_{\a,\b}\geq0$ and derive some Laplacian comparison type estimate. In Section \ref{sec:splitting}, we will prove the splitting theorem under $\mathcal{C}_{\a,\b}\geq0$, which is an extension of Cheeger-Gromoll's splitting theorem \cite{CheegerGromoll1971} in the \K case. In Section \ref{sec:vanishing}, we will prove some vanishing results under quasi-positive $\mathcal{C}_{\a,\b}$ or $\mathcal{C}_{\a,\b}\geq0$ and special holonomy group. In Section \ref{sec:proof and applications}, we will prove Theorem \ref{thm:structure} and \ref{thm:quasi-positive-mixed}, and apply them to \K manifolds with various known curvature conditions. In Appendix \ref{sec:conformal calculation}, we will collect some known formulas in the conformal geometry.

\bigskip

{\it Acknowledgement:} The authors would like to thank Valentino Tosatti for useful discussion. The first-named author was partially supported by National Key R\&D Program of China 2023YFA1009900 and NSFC grant 12271008. The second-named author was partially supported by Hong Kong RGC grant (Early Career Scheme) of Hong Kong No. 24304222, No. 14300623, and a NSFC grant No. 12222122. The third-named author was partially supported by National Key R\&D Program of China with grant no. 2023YFA1009900 as well as the startup fund from Westlake University.

\section{Mixed curvature for non-K\"ahler manifolds}\label{sec:mixed curvature}

In this section, we will extend the notion of mixed curvature to general Hermitian manifolds, especially to non-K\"ahler manifolds. We start with recalling some definitions in Hermitian geometry. Let $(M,g,J)$ be a Riemannian manifold with a complex structure $J$. The triple $(M,g,J)$ is said to be Hermitian if $g(Jv,Jw) = g(v,w)$ for all $v,w\in TM$. In this case, we might define the $2$-form $\omega$ by
\[
\omega(v,w) = g(Jv,w) \ \ \text{for $v,w\in TM$}.
\]
Using $J$, the complexified tangent space $T_{\mathbb{C}}M$ can be split as $TM=T^{1,0}M\oplus T^{0,1}M$, where $T^{1,0}M$ and $T^{0,1}M$ are eigenspaces of $J$ corresponding to eigenvalues $\sqrt{-1}$ and $\sqrt{-1}$ respectively. For any unitary $(1,0)$-vector $X$, $v=\frac{1}{\sqrt{2}}(X+\bar{X})$ is a unit vector such that
\[
X = \frac{1}{\sqrt{2}}(v-\sqrt{-1}Jv).
\]
The triple $(M,g,J)$ is said to be K\"ahler if $\omega$ is $d$-closed. An equivalent characterization of the K\"ahler condition is that the complex structure $J$ is parallel with respect to the Levi-Civita connection $\nabla$ of $g$ (i.e. $\nabla J=0$).

When $(M,g,J)$ is K\"ahler, let us re-formulate $\mathcal{C}_{\a,\b}$ in term of Riemannian geometry. Using $\nabla J=0$, direct calculation shows
\[
\mathcal{C}_{\alpha,\beta}(X) = \alpha\Ric(v,v)+\beta R(v,Jv,Jv,v).
\]
For convenience, we  write
\[
\mathcal{C}_{\alpha,\beta}(v) = \alpha\Ric(v,v)+\beta R(v,Jv,Jv,v)
\]
for $v\in TM$ with $|v|=1$. Such formulation can be extended directly to general Hermitian manifolds.

\begin{defn}\label{def:c a b Hermitian}
Let $(M,g,J)$ be a Hermitian manifold. Define
\[
\mathcal{C}_{\alpha,\beta}(v) = \alpha\Ric(v,v)+\beta R(v,Jv,Jv,v)
\]
for $v\in TM$ with $|v|=1$, where $\Ric(v,v)$ and $R(v,Jv,Jv,v)$ denote the Ricci and sectional curvature of the Levi-Civita connection. We say that $\mathcal{C}_{\a,\b}\geq 0$ (resp. $\mathcal{C}_{\a,\b}>0$) if $\mathcal{C}_{\a,\b,g}(v)\geq 0$ (resp. $\mathcal{C}_{\a,\b,g}(v)>0$) for all $x\in M$ and unit $v\in T_xM$.
\end{defn}

\begin{rem}
In the above definition, the Ricci and sectional curvature are of the Levi-Civita connection. This is different from \cite[Section 6]{ChuLeeTam2022} where the Chern connection is considered instead. When $(M,g,J)$ is K\"ahler, these two definitions are identical, since the Levi-Civita connection coincides with the Chern connection in this case.
\end{rem}

\section{Geodesic line}\label{sec:geodesic line}

In this section, we will show $\Ric(\gamma',\gamma')=0$ for any geodesic line $\gamma$ on K\"ahler manifolds with $\mathcal{C}_{\a,\b}\geq 0$, and then derive some Laplacian comparison type estimate along $\gamma$. These play an important role in the proof of splitting theorem (Theorem \ref{thm:splitting}), especially verifying the harmonicity of the Busemann functions. Throughout this paper, all geodesics are parametrized by arc-length for convenience.

\begin{prop}\label{prop:geodesic line curvature}
Let $(M^{n},g,J)$ be a complete K\"ahler manifold with $\mathcal{C}_{\alpha,\beta}\geq0$ where $\alpha>0$ and $\alpha+\beta> 0$, and $\gamma:(-\infty,+\infty)\to M$ be a geodesic line.
Then $R(v,\gamma',\gamma',v)=0$ for any $v\perp\gamma'$. In particular, $\Ric(\gamma',\gamma')=0$ along $\gamma$.
\end{prop}

\begin{proof}
By re-parametrization, it suffices to prove the conclusion at $\gamma(0)$. For any $\ell>0$, since $\gamma$ is a geodesic line, then $\gamma|_{[-\ell,\ell]}$ is minimizing geodesic connecting $\gamma(-\ell)$ and $\gamma(\ell)$. Let $V(t)$ be a vector field along $\gamma|_{[-\ell,\ell]}$ such that
\begin{equation}\label{geodesic line curvature eqn 1}
V\perp\gamma', \ \ V(-\ell) = 0, \ \ V(\ell) = 0.
\end{equation}
Define
\[
\phi(t,s) = \mathrm{exp}_{\gamma(t)}(sV(t)), \ \
L(s) = \mathrm{Length}(\phi(\cdot,s)).
\]
Since $\gamma|_{[-\ell,\ell]}$ is minimizing, then $L(s)$ achieves its minimum at $s=0$. The second variation formula of arc length gives:
\[
0 \leq \frac{d^{2}}{ds^{2}}\bigg|_{s=0}L(s)
=  \int_{-\ell}^{\ell}\left(|\nabla V|^{2}-R(V,\gamma',\gamma',V)\right)dt.
\]
For convenience, we introduce the notation:
\[
I_{\ell}(V) = \int_{-\ell}^{\ell}\left(|\nabla V|^{2}-R(V,\gamma',\gamma',V)\right)dt,
\]
which is non-negative  for all $V$ satisfying \eqref{geodesic line curvature eqn 1}.

Let $\{e_{i}\}_{i=1}^{2n}$ be an orthonormal parallel vector fields along $\gamma$ with
\[
e_{1} = \gamma', \ \ e_{2} = J\gamma'.
\]
To prove Proposition \ref{prop:geodesic line curvature}, it suffices to show
\begin{equation}\label{geodesic line curvature goal}
R\big(e_{i}(0),\gamma'(0),\gamma'(0),e_{i}(0)\big) = 0 \ \ \text{for $i=2,\cdots,2n$}.
\end{equation}
For $\ell\geq2$, let $\eta_{\ell}:\mathbb{R}\to[0,+\infty)$ be an one-variable non-negative function such that
\begin{itemize}\setlength{\itemsep}{1mm}
\item $\eta_{\ell}(t)=1$ when $t\in[-1,1]$;
\item $\eta_{\ell}(t)=0$ when $t\in(-\infty,-\ell]\cup[\ell,+\infty)$;
\item $|\eta_{\ell}'(t)|\leq 2/\ell$.
\end{itemize}
Then we compute
\begin{equation}\label{geodesic line curvature eqn 2}
\begin{split}
& (\alpha+\beta)I_{\ell}(\eta_{\ell}e_{2})+\alpha\sum_{i=3}^{2n}I_{\ell}(\eta_{\ell}e_{i}) \\
= {} & \Big((\alpha+\beta)+(2n-2)\alpha\Big)\int_{-\ell}^{\ell}|\eta_{\ell}'(t)|^{2}dt \\
& -\int_{-\ell}^{\ell}\eta_{\ell}^{2}\Big((\alpha+\beta)R(J\gamma',\gamma',\gamma',J\gamma')+\alpha\sum_{i=3}^{2n}R(e_{i},\gamma',\gamma',e_{i})\Big)\,dt \\
\leq {} & C(n,\a,\b)\cdot\frac{8}{\ell}-\int_{-\ell}^{\ell}\eta_{\ell}^{2}\cdot\mathcal{C}_{\alpha,\beta}(\gamma')\,dt\\[2mm]
\leq {} & C(n,\a,\b)\ell^{-1},
\end{split}
\end{equation}
where we used $\mathcal{C}_{\alpha,\beta}\geq0$ in the last line. On the other hand, the vector field $\eta_{\ell}e_{i}$ ($i=2,\cdots,2n$) satisfies \eqref{geodesic line curvature eqn 1} and so
\begin{equation}\label{geodesic line curvature eqn 3}
I_{\ell}(\eta_{\ell}e_{i}) \geq 0 \ \ \text{for $i=2,\cdots,2n$}.
\end{equation}
Using \eqref{geodesic line curvature eqn 2}, \eqref{geodesic line curvature eqn 3}, $\alpha>0$ and $\alpha+\beta>0$, we conclude that
\begin{equation}\label{geodesic line curvature eqn 4}
0 \leq I_{\ell}(\eta_{\ell}e_{i}) \leq C(n,\a,\b)\ell^{-1} \ \ \text{for $i=2,\cdots,2n$}.
\end{equation}

For any $\varphi\in C_{0}^{\infty}([-1,1])$ and $\tau\in(-1,1)$, $\eta_{\ell}e_{i}+\tau\varphi e_{i}$ ($i=2,\cdots,2n$) is a vector field satisfying \eqref{geodesic line curvature eqn 1} and hence,
\[
\begin{split}
0 \leq {} & I_{\ell}(\eta_{\ell}e_{i}+\tau\vp e_{i}) \\
= {} & I_{1}(\eta_{l}e_{i}+\tau\vp e_{i})+I_{\ell}(\eta_{l}e_{i}+\tau\vp e_{i})-I_{1}(\eta_{l}e_{i}+\tau\vp e_{i}) \\
= {} & I_{1}(\eta_{l}e_{i}+\tau\vp e_{i})+I_{\ell}(\eta_{l}e_{i})-I_{1}(\eta_{l}e_{i}) \\
= {} & I_{1}(e_{i}+\tau\vp e_{i})+I_{\ell}(\eta_{\ell}e_{i})-I_{1}(e_{i}).
\end{split}
\]
Letting $\ell\to+\infty$ and using \eqref{geodesic line curvature eqn 4}, we obtain
\[
I_{1}(e_{i}+\tau\vp e_{i}) \geq I_{1}(e_{i}).
\]
This implies
\[
0 = \frac{d}{d\tau}\bigg|_{\tau=0}I_{1}(e_{i}+\tau\vp e_{i})
= -2\int_{-1}^{1}\vp R(e_{i},\gamma',\gamma',e_{i})dt
\]
for all $\vp\in C_{0}^{\infty}([-1,1])$. Then we obtain \eqref{geodesic line curvature goal}.
\end{proof}

As an immediate consequence, we have the following Laplacian comparison type estimate along the geodesic line.
\begin{cor}\label{Laplacian comparison}
Under the same assumptions of Proposition \ref{prop:geodesic line curvature}, for each $t\in\mathbb{R}$, $r(x)=d(x,\gamma(t))$ is a function on $M$ satisfying
\[
\Delta r \leq \frac{2n-1}{r} \ \ \text{for $x\in \gamma\setminus\{\gamma(t)\}$}
\]
whenever $r(\cdot)=d(\cdot,\gamma(t))$ is smooth.
\end{cor}

\begin{proof}
The argument is standard. For the reader's convenience, we include a sketch here. By re-parametrization, we may assume that $t=0$ and so $r(x)=d(\cdot,\gamma(0))$. For $x\in\gamma\setminus\{\gamma(0)\}$, $\gamma$ is the minimizing geodesic connecting $\gamma(0)$ and $x$. By the Bochner formula, Proposition \ref{prop:geodesic line curvature} and $\dim_{\mathbb{R}}M=2n$,
\[
\begin{split}
0 = {} & \frac{1}{2}\Delta |\nabla r|^{2} \\[1mm]
= {} & |\nabla^{2}r|^{2}+g(\nabla\Delta r,\nabla r)+\Ric(\nabla r,\nabla r) \\
= {} & |\nabla^{2}r|^{2}+\frac{d}{dr}\Delta r+\Ric(\gamma',\gamma') \\
\geq {} & \frac{(\Delta r)^{2}}{2n-1}+\frac{d}{dr}\Delta r.
\end{split}
\]
Together with the initial condition $\lim_{r\to0}\Delta r^{2}=4n$, we are done.
\end{proof}

\section{Splitting theorem}\label{sec:splitting}

In this section, we aim to prove a splitting theorem for \K manifolds with $\mathcal{C}_{\a,\b}\geq0$, provided that it contains a geodesic line.  This is analogous to Cheeger-Gromoll's splitting theorem \cite{CheegerGromoll1971} in the purely Riemannian case with $\Ric\geq 0$.

\begin{thm}\label{thm:splitting}
Let $(M^{n},g,J)$ be a complete K\"ahler manifold with $\mathcal{C}_{\alpha,\beta}\geq0$ where $\alpha>0$ and $\alpha+\beta>0$. If there exists a geodesic line $\gamma_{0}:(-\infty,+\infty)\to M$, then $M$ splits isometrically as $N\times \mathbb{R}$ as a Riemannian manifold. In particular, its universal cover $\tilde M$ splits isometrically as $\tilde M'\times\mathbb{C}$.
\end{thm}

Following the spirit of Cheeger-Gromoll's splitting theorem \cite{CheegerGromoll1971}, it suffices to find a function $f$ on $M$ such that $|\nabla f|=1$ and $\nabla^2f\equiv 0$ by using the Busemann function. For the given geodesic line $\gamma_0:(-\infty,+\infty)\to M$, we define the associated Busemann function by
\begin{equation*}
\begin{cases}
\ \mathcal{B}_+(\cdot) = \displaystyle{\lim_{s\to +\infty}} \big(d_g(\cdot,\gamma_0(s))-s \big); \\[2mm]
\ \mathcal{B}_-(\cdot) = \displaystyle{\lim_{s\to +\infty}} \big(d_g(\cdot,\gamma_0(-s))-s \big).
\end{cases}
\end{equation*}
In the presence of $\Ric\geq 0$, it follows from the classical Laplacian comparison that $\mathcal{B}_\pm$ are both harmonic which serves as our desired function $f$. However, in our situation, the Laplacian comparison type estimate only holds along the geodesic line (see Corollary \ref{Laplacian comparison}). So we need the following geodesic line foliation lemma.

\begin{prop}\label{prop:technical proposition}
Under the same assumptions of Theorem \ref{thm:splitting}, suppose $\gamma:(-\infty,+\infty)\to M$ is a geodesic line (which is possibly different from $\gamma_0$) such that
\begin{enumerate}\setlength{\itemsep}{1mm}
\item[(i)] $\mathcal{B}_{+}+\mathcal{B}_{-}=0$ on $\gamma$;
\item[(ii)] $\mathcal{B}_{+}(\gamma(t_{1}))-\mathcal{B}_{+}(\gamma(t_{2}))=t_{2}-t_{1}$ for all $t_1,t_2\in\mathbb{R}$,
\end{enumerate}
and $\sigma:[0,+\infty)\to M$ be a geodesic such that
\[
\sigma(0) = \gamma(0), \ \ \sigma'(0)\perp\gamma'(0).
\]
Then there exist constants $\mu_{0},r_{0}\in(0,1)$ depending only on $n$, $\a$, $\b$ and the geometry at $\gamma(0)$ such that the following holds. For any $r<r_{0}$, there is a geodesic line $\hat{\gamma}_{r}$ such that
\begin{enumerate}\setlength{\itemsep}{1mm}
\item[(a)] $\hat{\gamma}_{r} \cap \overline{B_{(1-\mu_{0})r}(\sigma(r))}\neq\emptyset$;
\item[(b)] $\mathcal{B}_{+}+\mathcal{B}_{-}=0$ on $\hat{\gamma}_{r}$;
\item[(c)] $\mathcal{B}_{+}(\hat{\gamma}_{r}(t_{1}))-\mathcal{B}_{+}(\hat{\gamma}_{r}(t_{2}))=t_{2}-t_{1}$ for all $t_1,t_2\in \mathbb{R}$.
\end{enumerate}
\end{prop}

We first prove Theorem~\ref{thm:splitting} assuming Proposition \ref{prop:technical proposition} holds.
\begin{proof}[Proof of Theorem~\ref{thm:splitting} assuming Proposition~\ref{prop:technical proposition}]

We first show that for any $p\in M$, there exists a geodesic line $\gamma_{p}$ passing through $p$ such that
\begin{enumerate}\setlength{\itemsep}{1mm}
\item[(i)] $\mathcal{B}_{+}+\mathcal{B}_{-}=0$ on $\gamma_p$;
\item[(ii)] $\mathcal{B}_{+}(\gamma_{p}(t_{1}))-\mathcal{B}_{+}(\gamma_{p}(t_{2}))=t_{2}-t_{1}$ for all $t_1,t_2\in \mathbb{R}$.
\end{enumerate}
Define
\[
\mathbf{\Gamma} = \{\gamma: \text{$\gamma$ is a geodesic line on $M$ such that (i) and (ii) hold} \}
\]
and
\[
i_{p} = \inf\{d_g(p,\gamma):\gamma\in \mathbf{\Gamma}\}.
\]
Since $\gamma_0\in\mathbf{\Gamma}$, then $\mathbf{\Gamma}$ is non-empty and $i_{p}$ is well-defined. By the compactness of geodesic line, there exists a geodesic line $\gamma\in\mathbf{\Gamma}$ such that $
d_g(p,\gamma) = i_{p}$. Suppose $i_p>0$, it follows from  Proposition \ref{prop:technical proposition} that there exists another geodesic line $\hat\gamma\in\mathbf{\Gamma}$ such that $d_g(p,\hat\gamma) < i_{p}$ which is impossible. Then we obtain $i_p=0$, which implies the existence of $\gamma_{p}$.

Next we show $\Delta \mathcal{B}_{\pm}\leq0$ in the barrier sense on $M$. For any $p$, by re-parameterization, we may assume that $\gamma_{p}(0)=p$. For $t>0$, define
\[
\begin{cases}
\ f_{+,t}(x) = \mathcal{B}_{+}(p)-t+d_g(x,\gamma_{p}(t)); \\[1mm]
\ f_{-,t}(x) = \mathcal{B}_{-}(p)-t+d_g(x,\gamma_{p}(-t)).
\end{cases}
\]
It is clear that
\begin{equation}\label{upper barrier eqn 1}
\begin{cases}
\ f_{+,t}(p) = \mathcal{B}_{+}(p); \\[1mm]
\ f_{-,t}(p) =  \mathcal{B}_{-}(p).
\end{cases}
\end{equation}
Since $\gamma_{p}$ satisfies (i) and (ii), together with the triangle inequality, we see that
\begin{equation}\label{upper barrier eqn 2}
\begin{cases}
\ f_{+,t}(x) =  \mathcal{B}_{+}(\gamma_{p}(t))+d_g(x,\gamma_{p}(t)) \geq  \mathcal{B}_{+}(x); \\[1mm]
\ f_{-,t}(x) = \mathcal{B}_{-}(\gamma_{p}(-t))+d_g(x,\gamma_{p}(-t)) \geq  \mathcal{B}_{-}(x).
\end{cases}
\end{equation}
The above show that $f_{\pm,t}$ is a upper barrier for $\mathcal{B}_{\pm}$ at $p$. Thanks to Corollary \ref{Laplacian comparison},
\[
\Delta f_{\pm,t}(p) \leq \frac{2n-1}{t}.
\]
Then for any $\e>0$, after choosing $t$ sufficiently large, we obtain $\Delta f_{\pm,t}(p)<\e$. This implies $\Delta \mathcal{B}_{\pm}\leq0$ in the barrier sense on $M$. Since $\mathcal{B}_{+}+\mathcal{B}_{-}$ is non-negative on $M$ (follows from the triangle inequality) and attains $0$ along $\gamma_p$, it then follows from the strong maximum principle that $\mathcal{B}_{+}+\mathcal{B}_{-}\equiv 0$ on $M$ and thus $\mathcal{B}_+$ is smoothly harmonic on $M$. Together with \eqref{upper barrier eqn 1} and \eqref{upper barrier eqn 2}, we have
\[
f_{+,t}(x) \geq \mathcal{B}_{+}(x) = -\mathcal{B}_{-}(x) \geq -f_{-,t}(x)
\]
and the equality holds when $x=p$. This means
\[
\nabla f_{+,t}(p) = \nabla \mathcal{B}_{+}(p) = -\nabla \mathcal{B}_{-}(p) = -\nabla f_{-,t}(p)
\]
so that $|\nabla \mathcal{B}_{+}|(p)=1$ as $\nabla \mathcal{B}_{+}(p)=\nabla f_{+,t}(p)=-\gamma_p'(0)$. Since $p$ is arbitrary,   $|\nabla \mathcal{B}_{+}|\equiv 1$ on $M$. Applying the Bochner formula with Proposition \ref{prop:geodesic line curvature} and evaluating at $p$, we see that
\[
\begin{split}
0 = {} & \frac{1}{2}\Delta |\nabla  \mathcal{B}_{+}|^{2} \\[0.5mm]
= {} & |\nabla^{2} \mathcal{B}_{+}|^{2}+g(\nabla\Delta \mathcal{B}_{+},\nabla  \mathcal{B}_{+})+\Ric(\nabla  \mathcal{B}_{+},\nabla  \mathcal{B}_{+}) \\[1mm]
= {} & |\nabla^{2} \mathcal{B}_{+}|^{2}+\Ric(\gamma_{p}',\gamma_{p}') = |\nabla^{2} \mathcal{B}_{+}|^{2}.
\end{split}
\]
It then follows that $\nabla^{2}\mathcal{B}_{+}=0$ at $p$. Since $p$ is arbitrary, $\nabla^{2} \mathcal{B}_{+}\equiv0$ on $M$. This implies that $M$ must split as $N\times \mathbb{R}_1$ isometrically.

Since $\nabla \mathcal{B}_+$ is a parallel vector field (corresponding the Euclidean factor $\mathbb{R}_1$), the K\"ahlerity of $(M,g)$ implies $J\nabla \mathcal{B}_+$ is a parallel vector field on $N$. It follows from the de Rham decomposition theorem \cite{deRham1952} that the universal cover of $N$ splits off another Euclidean factor $\mathbb{R}_2$. Hence, $\tilde M$ splits isometrically as $\tilde{M}'\times \mathbb{C}$ where $\mathbb{C}=\mathbb{R}_1\oplus\sqrt{-1}\mathbb{R}_2$.
\end{proof}

The rest of this section will be devoted to prove technical Proposition \ref{prop:technical proposition}. This is inspired by the ideas in Liu \cite{Liu2013} and Chodosh-Eichmair-Moraru \cite{ChodoshEichmairMoraru2019}.

\begin{proof}[Proof of Proposition~\ref{prop:technical proposition}]

We will split the proof into several steps. We first fix $r_0<1$ sufficiently small (depending on the geometry at $\gamma(0)$) so that for all $r<r_0$, the function $d^2_g\left(\cdot, \sigma(r)\right)$ is smooth and satisfies $|\nabla^2 d^2_g(\cdot,\sigma(r))|\leq C_n$ in $B_{g}(\sigma(r),2r)$. Denote $z_r:=\sigma(r)$ which is $r$-distance away from $\sigma(0)=\gamma(0)$. The constant $r$ will be fixed (but arbitrary) throughout the proof, for notational convenience we omit the index $r$ and write $z=z_{r}$. We will also use the following notation to denote the annulus:
\[
A_{g}(p,r_{1},r_{2}) := B_{g}(p,r_{2})\setminus\overline{B_{g}(p,r_{1})}.
\]

In the following argument, a constant $C$ is said to be uniform if $C=C(n,\alpha,\beta)$. We will use $O(T)$ to denote a term satisfying $|O(T)|\leq CT$ for some uniform constant $C$.

We will fix the value of constant $\mu_{0}(n,\alpha,\beta)$ (which is uniform) in Claim \ref{C a b}. Note that any constant $C(n,\alpha,\beta,\mu_{0})$ is also uniform.

\bigskip
\noindent
{\bf Step 1.} Construct conformal metric $\hat{g}_{\e}$ with positive $\mathcal{C}_{\a,\b,\hat g_{\e}}$ somewhere.
\bigskip

We define $\rho(\cdot)=(1+\mu_{0})^2r^2-d^2_g(\cdot, z)$ which is smooth in $B_g(z,(1+\mu_{0})r)$, and let $f$ be the function on $\mathbb{R}$ given by
\[
f(s) =
\begin{cases}
\ e^{-1/s}  & \mbox{for $s>0$};\\
\ 0  & \mbox{for $s\leq0$}.
\end{cases}
\]
For notational convenience, function in form of $e^{-1/\rho}\rho^{-m}$ at $\rho=0$ is understood to be $0$.  Define the function $\varphi=f\circ \rho$ and metric $\hat g_{\e}=e^{-2\e\varphi}g$ for $\e>0$. By $f(s)\leq s^{2}$, each $\hat g_{\e}$ satisfies
\begin{equation}\label{eqn:metric-compare}
e^{-2\e (1+\mu_{0})^2 r^2}g\leq \hat g_{\e}\leq g.
\end{equation}

We next consider the mixed curvature of $\hat{g}_{\ve}$.
\begin{claim}\label{C a b}
There exists uniform constants $\mu_{0}\in (0,1/8)$ and $\Lambda_{0}>1$ such that
\[
\mathcal{C}_{\a,\b,\hat g_{\e}}\geq \e r^{2}\Lambda_{0}^{-1}\rho^{-4}e^{-1/\rho} \ \
\text{on $M\setminus B_{g}(z,(1-\mu_{0})r)$}.
\]
\end{claim}

\begin{proof}[Proof of Claim \ref{C a b}]
Outside $B_g( z, (1+\mu_{0})r)$, $\hat g_\e=g$ so that the conclusion holds trivially since $\mathcal{C}_{\a,\b,g}\geq 0$. Hence it suffices to consider the annulus $A_g(z,(1-\mu_{0})r,(1+\mu_{0})r)$. Write $h=\e\varphi$ and so $\hat{g}_{\e}=e^{-2h}g$. It is clear that
\[
\begin{cases}
\ \nabla h= \e\nabla \varphi=\e \rho^{-2}e^{-1/\rho} \nabla \rho; \\[1mm]
\ \nabla^{2}h = \e(1-2\rho)\rho^{-4} e^{-1/\rho} \nabla\rho\otimes \nabla\rho+\e \rho^{-2}e^{-1/\rho} \nabla^2\rho.
\end{cases}
\]
Write $d_{g}=d_{g}(\cdot,z)$ and so $\rho=(1+\mu_{0})^2r^2-d_{g}^{2}$. On annulus $A_g(z,(1-\mu_{0})r,(1+\mu_{0})r)$, we have
\begin{equation}\label{d g rho bounds}
(1-\mu_{0})^{2}r^{2} \leq d_{g}^{2} \leq (1+\mu_{0})^{2}r^{2} \ \, \text{and} \ \,
0\leq\rho\leq 4\mu_{0} r^2.
\end{equation}
It then follows from $\mu_{0}\in(0,1/8)$ and $r\leq r_0<1$ that $1-2\rho>0$. We also have $|\nabla\rho|+|\nabla^{2}\rho|\leq C$, and then
\begin{equation}\label{eqn:rho-estim}
\begin{cases}
\ |\nabla h|^{2} \leq C\e^{2}\rho^{-4}e^{-2/\rho}; \\[1mm]
\ \nabla^{2}h = 4\e  d_{g}^{2}\cdot(1-2\rho)\rho^{-4} e^{-1/\rho}\nabla d_{g}\otimes \nabla d_{g}+O(\e \rho^{-2}e^{-1/\rho}).
\end{cases}
\end{equation}
Write $\hat{v}=e^{h}v$ for some $g$-unit vector $v$. Combining \eqref{eqn:rho-estim} with Lemma \ref{Conformal lemma} (b) and using $\mathcal{C}_{\a,\b,g}\geq 0$, we obtain
\[
\begin{split}
& \mathcal{C}_{\alpha,\beta,\hat{g}_{\e}}(\hat{v}) \\[1mm]
\geq {} & e^{2h}\Big( \alpha\Delta h+(2n\alpha-2\alpha+\beta)\nabla^{2}h(v,v)+\beta\nabla^{2}h(Jv,Jv)-C|\nabla h|^{2} \Big) \\[1mm]
\geq {} & e^{2h}\Big( 4\ve d_{g}^{2}\cdot(1-2\rho)\rho^{-4}e^{-1/\rho}\big(\alpha+(2n\alpha-2\alpha+\beta)v(d_{g})^{2}+\beta Jv(d_{g})^{2}\big) \\
& \quad \quad -C\ve\rho^{-2}e^{-1/\rho}-C\ve^{2}\rho^{-4}e^{-2/\rho} \Big) \\
\geq {} & e^{2h}\Big( 4\ve d_{g}^{2}\cdot(1-2\rho)\rho^{-4}e^{-1/\rho}\big(\alpha+\beta Jv(d_{g})^{2}\big)-C\ve\rho^{-2}e^{-1/\rho} \Big),
\end{split}
\]
where we used $1-2\rho>0$ and $2n\alpha-2\alpha+\beta>0$ (follows from $n\geq2$, $\alpha>0$ and $\alpha+\beta>0$) in the last line.
If $\beta\geq0$, then $\alpha+\beta Jv(d_{g})^{2}\geq\alpha>0$. If $\beta<0$, then
\[
\alpha+\beta Jv(d_{g})^{2}
= (\alpha+\beta)-\beta\big(1-Jv(d_{g})^{2}\big) \geq \alpha+\beta > 0.
\]
In both cases, we have $\alpha+\beta Jv(d_{g})^{2}\geq C^{-1}$ and so
\[
\mathcal{C}_{\alpha,\beta,\hat{g}_{\e}}(\hat{v})
\geq  \ve e^{2h}\rho^{-4}e^{-1/\rho}\big( 4d_{g}^{2}\cdot(1-2\rho)\cdot C^{-1}-C\rho^{2}\big).
\]
Using \eqref{d g rho bounds} and choosing $\mu_{0}$ sufficiently small, we obtain
\[
\mathcal{C}_{\alpha,\beta,\hat{g}_{\e}}(\hat{v})
\geq \ve e^{2h}\rho^{-4}e^{-1/\rho}\big( C^{-1}r^{2}-C\mu_{0}^{2}r^{4}\big)
\geq \ve e^{2h}\rho^{-4}e^{-1/\rho}\cdot \frac{C^{-1}r^{2}}{2}.
\]
Combining this with $h=\e\varphi\geq0$, we complete the proof.
\end{proof}

\bigskip
\noindent
{\bf Step 2.} Construct geodesic lines $\hat \gamma_{\e,\ell}$ and $\hat \gamma$.
\bigskip

For any $\ell>1$, we let $\hat \gamma_{\e,\ell}$ be the $\hat g_\e$-minimizing geodesic connecting $\gamma(-\ell)$ and $\gamma(\ell)$. Then $\hat\gamma_{\e,\ell}$ must intersect $B_{g}(z,(1+\mu_{0})r)$. Indeed, if $\hat \gamma_{\e,\ell}$ does not intersect $B_{g}(z,(1+\mu_{0})r)$, using $\hat g_{\e}=g$ outside
$B_{g}(z,(1+\mu_{0})r)$, we see that
\[
\mathrm{Length}_{\hat g_\e}(\hat\gamma_{\e,\ell})
= \mathrm{Length}_{g}(\hat\gamma_{\e,\ell})
\geq d_g(\gamma(-\ell),\gamma(\ell))
= \mathrm{Length}_{g}(\gamma|_{[-\ell,\ell]}).
\]
However, $\gamma|_{[-\ell,\ell]}$ intersects the region where $\hat g_\e<g$, then
\[
\mathrm{Length}_{g}(\gamma|_{[-\ell,\ell]})
> \mathrm{Length}_{\hat g_\e}(\gamma|_{[-\ell,\ell]})
\geq d_{\hat g_{\e}}(\gamma(-\ell),\gamma(\ell))
= \mathrm{Length}_{\hat g_\e}(\hat\gamma_{\e,\ell}),
\]
which is a contradiction.

Suppose $\hat\gamma_{\e,\ell}$ is re-parametrized by arc-length, namely, $\hat\gamma_{\e,\ell}:[-\hat\ell,\hat \ell]\to M $ from $\hat\gamma_{\e,\ell}(-\hat\ell)=\gamma(-\ell)$ to $\hat\gamma_{\e,\ell}(\hat\ell)=\gamma(\ell)$. Then using \eqref{eqn:metric-compare} and $\hat g_\e=g$ outside $B_g(z,(1+\mu_{0})r)$, we see that
\begin{equation*}
\begin{split}
2\ell \geq  2\hat \ell = {} & \mathrm{Length}_{\hat g_\e}(\hat\gamma_{\e,\ell}) \\[1.5mm]
\geq {} & \mathrm{Length}_{ g}(\hat\gamma_{\e,\ell})
+\int_{\hat \gamma_{\e,\ell}\cap B_{g}(z,(1+\mu_{0})r)} (|\hat\gamma_{\e,\ell}'|_{\hat g_\e}-|\hat\gamma_{\e,\ell}'|_{g}) dt \\
\geq {} & 2\ell+\left(1-e^{\e(1+\mu_{0})^2r^2} \right)\cdot\int_{\hat \gamma_{\e,\ell}\cap B_{g}(z,(1+\mu_{0})r)} |\hat\gamma_{\e,\ell}'|_{\hat g_\e}\, dt \\
= {} & 2\ell-\left(e^{\e(1+\mu_{0})^2r^2}-1 \right)\cdot\mathrm{Length}_{\hat g_\e}\big(\hat \gamma_{\e,\ell}\cap B_{g}(z,(1+\mu_{0})r)\big) \\[1mm]
\geq {} & 2\ell-\left(e^{\e(1+\mu_{0})^2r^2}-1 \right)\cdot\mathrm{Diam}_{\hat{g}_{\ve}}\big(B_{g}(z,(1+\mu_{0})r)\big) \\[1mm]
\geq {} & 2\ell-\left(e^{\e(1+\mu_{0})^2r^2}-1 \right)\cdot\mathrm{Diam}_{g}\big(B_{g}(z,(1+\mu_{0})r)\big) \\[1mm]
\geq {} & 2\ell-\left(e^{\e(1+\mu_{0})^2r^2}-1 \right)\cdot 2(1+\mu_{0})r.
\end{split}
\end{equation*}
This shows $|\hat{\ell}-\ell|\leq C\e$ for some uniform constant $C$.

Since $\hat\gamma_{\e,\ell}$ intersects $B_{g}(z,(1+\mu_{0})r)$, then the compactness of geodesic implies $\hat\gamma_{\e,\ell}\to \hat \gamma_{\e}$ for some subsequence $\ell\to+\infty$ and $\hat g_\e$-geodesic line $\hat \gamma_{\e}$. Passing to a subsequence of $\e\to0$ again, we obtain a geodesic line $\hat\gamma=\lim_{\e\to 0}\hat\gamma_{\e}$ with respect to metric $g=\lim_{\e\to 0}\hat g_\e$. In the rest of the proof, we will show that the constructed geodesic $\hat \gamma$ satisfies the required properties (a)-(c) in Proposition \ref{prop:technical proposition}.

\bigskip
\noindent
{\bf Step 3.} $d_{g}(z,\hat{\gamma}_{\e,\ell})<(1+\mu_{0}-\delta_{0})r$ for some uniform constant $\delta_{0}>0$.
\bigskip

We will prove it by contradiction and specify our choice of $\delta_{0}$ later. Assume $d_{g}(z,\hat{\gamma}_{\e,\ell})\geq (1+\mu_{0}-\delta_{0})r$, then
\begin{equation}\label{eqn:rho-away}
\rho(\cdot)\leq \rho_1 := (1+\mu_{0})^2r^2-(1+\mu_{0}-\delta_{0})^2r^2 \ \ \text{on  $\hat \gamma_{\e,\ell}\cap B_g(z,(1+\mu_{0})r)$}.
\end{equation}
Since $\hat g_\e=g$ outside $B_g(z,(1+\mu_{0})r)$, we compute
\begin{equation}\label{eqn:length}
\begin{split}
& \mathrm{Length}_{\hat g_\e}(\hat \gamma_{\e,\ell}) \\[2mm]
= {} & \mathrm{Length}_{g}(\hat \gamma_{\e,\ell})
+\int_{\hat \gamma_{\e,\ell}\cap B_g(z,(1+\mu_{0})r)} (|\hat \gamma_{\e,\ell}'|_{\hat g_\e}-|\hat\gamma_{\e,\ell}'|_g)dt\\
\geq {} & d_g(\gamma(-\ell),\gamma(\ell))
+\int_{\hat \gamma_{\e,\ell}\cap B_g(z,(1+\mu_{0})r)} (|\hat \gamma_{\e,\ell}'|_{\hat g_\e}-|\hat\gamma_{\e,\ell}'|_g)dt\\
= {} & \mathrm{Length}_{g}( \gamma|_{[-\ell,\ell]})
+\int_{\hat \gamma_{\e,\ell}\cap B_g(z,(1+\mu_{0})r)} (|\hat \gamma_{\e,\ell}'|_{\hat g_\e}-|\hat\gamma_{\e,\ell}'|_g)dt.
\end{split}
\end{equation}
By interchanging the role of pairs $(\hat g_\e,\hat\gamma_{\e,\ell})$ and $(g,\gamma|_{[-\ell,\ell]})$,
\[
\mathrm{Length}_{g}( \gamma|_{[-\ell,\ell]})
\geq \mathrm{Length}_{\hat g_\e}(\hat \gamma_{\e,\ell})
+\int_{\gamma|_{[-\ell,\ell]}\cap B_g\left(z,(1+\mu_{0})r\right)}(|\gamma'|_{g}-|\gamma'|_{\hat g_\e})dt.
\]
Combining this with \eqref{eqn:length}, we conclude that
\begin{equation}\label{eqn:length-compare}
\int_{\hat \gamma_{\e,\ell}\cap B_g(z,(1+\mu_{0})r)} (|\hat\gamma_{\e,\ell}'|_g-|\hat \gamma_{\e,\ell}'|_{\hat g_\e})dt
\geq  \int_{ \gamma|_{[-\ell,\ell]}\cap B_g(z,(1+\mu_{0})r)}(|\gamma'|_{g}- |\gamma'|_{\hat g_\e})dt.
\end{equation}
We estimate the left-hand side of \eqref{eqn:length-compare} as follows. It follows from \eqref{eqn:rho-away} and $\hat g_\e=e^{-2\ve\vp}g\leq g$ that
\begin{equation}\label{eqn:length-compare LHS}
\begin{split}
& \int_{\hat \gamma_{\e,\ell}\cap B_g(z,(1+\mu_{0})r)} (|\hat\gamma_{\e,\ell}'|_g-|\hat \gamma_{\e,\ell}'|_{\hat g_\e})dt \\
= {} & \int_{\hat \gamma_{\e,\ell}\cap B_g(z,(1+\mu_{0})r)} (e^{\e \varphi}-1)|\hat \gamma_{\e,\ell}'|_{\hat g_\e}dt \\[3mm]
\leq {} & \left(e^{\e f(\rho_1)}-1\right)\cdot \mathrm{Length}_{\hat g_\e}\big(\hat \gamma_{\e,\ell}\cap B_g(z,(1+\mu_{0})r)\big) \\[3mm]
\leq {} & \left(e^{\e f(\rho_1)}-1\right)\cdot \mathrm{Diam}_{\hat{g}_{\ve}}\big(B_g(z,(1+\mu_{0})r)\big) \\[3mm]
\leq {} & \left(e^{\e f(\rho_1)}-1\right)\cdot \mathrm{Diam}_{g}\big(B_g(z,(1+\mu_{0})r)\big) \\[3mm]
\leq {} & \left(e^{\e f(\rho_1)}-1\right)\cdot 2(1+\mu_{0})r.
\end{split}
\end{equation}
On the other hand, by $d_{g}(\gamma(0),z)=r$, it is clear that $B_g(\gamma(0),\mu_{0} r/2)\subset B_g(z,(1+\mu_{0}/2)r)$ and
\[
\rho(\cdot) \geq \rho_{2} := (1+\mu_{0})^2r^2-(1+\mu_{0}/2)^2r^2 \ \ \text{on $B_g(\gamma(0),\mu_{0} r/2)$}.
\]
Then we estimate the right-hand side of \eqref{eqn:length-compare}:
\begin{equation}\label{eqn:length-compare RHS}
\begin{split}
& \int_{ \gamma|_{[-\ell,\ell]}\cap B_g(z,(1+\mu_{0})r)}(|\gamma'|_{g}- |\gamma'|_{\hat g_\e})dt \\
\geq {} & \int_{ \gamma|_{[-\ell,\ell]}\cap B_g(\gamma(0),\mu_{0}r/2)}(1-e^{-\e \varphi})|\gamma'|_{g}dt \\[1.5mm]
\geq {} & \left(1-e^{-\e f(\rho_{2})}\right)\cdot\mathrm{Length}_{g}\big(\gamma|_{[-\ell,\ell]}\cap B_g(\gamma(0), \mu_{0} r/2)\big) \\[3mm]
= {} & \left(1-e^{-\e f(\rho_{2})}\right)\cdot\mu_{0} r.
\end{split}
\end{equation}
Substituting \eqref{eqn:length-compare LHS} and \eqref{eqn:length-compare RHS} into \eqref{eqn:length-compare},
\begin{equation*}
\left(1-e^{-\e f(\rho_{2})}\right)\mu_{0} \leq 2(1+\mu_{0})\left(e^{\e f(\rho_1)}-1\right).
\end{equation*}
Using $1-e^{-x}\geq x/3$ and $e^{x}-1\leq 3x$ for $x\in[0,1]$, we obtain $f(\rho_{2})\mu_{0}\leq18(1+\mu_{0})f(\rho_{1})$. Together with the definitions of $f$, $\rho_{1}$ and $\rho_{2}$, we obtain
\[
-\frac{\mu_{0}}{(1+\mu_{0})^2r^2-(1+\mu_{0}/2)^2r^2}
\leq -\frac{18(1+\mu_{0})}{(1+\mu_{0})^2r^2-(1+\mu_{0}-\delta_{0})^2r^2},
\]
which is impossible if $\delta_{0}$ is sufficiently small depending only on $\mu_{0}$. This finishes the proof of {\bf Step 3}.

\medskip

Since $d_g(z,\hat \gamma_{\e,\ell})<(1+\mu_{0}-\delta_{0})r$, there exists $\hat a_{\e,\ell}$ such that $d_g(z,\hat \gamma_{\e,\ell})$ is realized at $\hat\gamma_{\e,\ell}(\hat a_{\e,\ell})$, i.e.
\[
d_{g}(z,\hat{\gamma}_{\e,\ell}(\hat{a}_{\e,\ell})) = d_g(z,\hat \gamma_{\e,\ell}) < (1+\mu_{0}-\delta_{0})r.
\]
For later use, we have the following estimate of $\hat{a}_{\e,\ell}$.

\begin{claim}\label{hat a claim}
$|\hat{a}_{\e,\ell}|\leq A_{0}$ for some uniform constant $A_{0}$.
\end{claim}

\begin{proof}[Proof of Claim \ref{hat a claim}]
When $\hat{a}_{\e,\ell}\geq0$, by the triangle inequality, $\hat{g}_{\ve}\leq g$ and $\hat{\gamma}_{\e,\ell}(\hat{\ell})=\gamma_{\e,\ell}(\ell)$, we obtain
\[
\begin{split}
\hat{\ell}-\hat{a}_{\e,\ell} = {} & d_{\hat{g}_\e}(\hat{\gamma}_{\e,\ell}(\hat{\ell}),\hat{\gamma}_{\e,\ell}(\hat{a}_{\e,\ell})) \\
\geq {} & d_{\hat{g}_\e}(\hat{\gamma}_{\e,\ell}(\hat{\ell}),z)-d_{\hat{g}_\e}(z,\hat{\gamma}_{\e,\ell}(\hat{a}_{\e,\ell})) \\
\geq {} & d_{\hat{g}_\e}(\gamma(\ell),z)-d_{g}(z,\hat{\gamma}_{\e,\ell}(\hat{a}_{\e,\ell})) \\
> {} & d_{\hat{g}_\e}(\gamma(\ell),z)-(1+\mu_{0}-\delta_{0})r.
\end{split}
\]
Since $\hat{g}_{\e}=g$ outside $B_g(z,(1+\mu_{0})r)$ and $\gamma(0)\in B_g(z,(1+\mu_{0})r)$, then
\[
\begin{split}
d_{\hat{g}_{\e}}(\gamma(\ell),z)
\geq {} & d_{g}(\gamma(\ell),B_g(z,(1+\mu_{0})r)) \\
\geq {} & d_{g}(\gamma(\ell),\gamma(0))-2(1+\mu_{0})r \\
= {} & \ell-2(1+\mu_{0})r.
\end{split}
\]
Combining the above with $|\hat{\ell}-\ell|\leq C\e$ and $r\leq r_{0}<1$, we obtain $|\hat{a}_{\e,\ell}|\leq A_{0}$. When $\hat{a}_{\e,\ell}<0$, the similar argument shows
\[
\begin{split}
\hat{\ell}+\hat{a}_{\e,\ell} = {} & d_{\hat{g}_\e}(\hat{\gamma}_{\e,\ell}(-\hat{\ell}),\hat{\gamma}_{\e,\ell}(\hat{a}_{\e,\ell})) \\
\geq {} & d_{\hat{g}_\e}(\hat{\gamma}_{\e,\ell}(-\hat{\ell}),z)-d_{\hat{g}_\e}(z,\hat{\gamma}_{\e,\ell}(\hat{a}_{\e,\ell})) \\
> {} & d_{\hat{g}_\e}(\gamma(-\ell),z)-(1+\mu_{0}-\delta_{0})r \\
\geq {} & \ell-2(1+\mu_{0})r-(1+\mu_{0}-\delta_{0})r,
\end{split}
\]
which also implies $|\hat{a}_{\e,\ell}|\leq A_{0}$.
\end{proof}

\bigskip
\noindent
{\bf Step 4.} $d_g(z,\hat \gamma_{\e,\ell})<(1-\mu_{0})r$ for all sufficiently large $\ell$ (depending only on $\e$, $r$, $n$, $\alpha$ and $\beta$).
\bigskip

We prove it by contradiction again. Suppose $d_g(z,\hat \gamma_{\e,\ell})\geq (1-\mu_{0})r$. Then
\begin{equation}\label{upper bound of rho}
\rho(\cdot) \leq (1+\mu_{0})r^{2}-(1-\mu_{0})r^{2} = 4\mu_{0}r^{2} \ \ \text{on $\hat{\gamma}_{\e,\ell}$}
\end{equation}
and Claim \ref{C a b} implies
\begin{equation}\label{lower bound of C a b}
\mathcal{C}_{\a,\b,\hat g_{\e}}\geq \e r^{2}\Lambda_{0}^{-1}\rho^{-4}e^{-1/\rho} \ \ \text{on $\hat{\gamma}_{\e,\ell}$}.
\end{equation}

Denote the Levi-Civita connection and curvature tensor of $\hat{g}_\e$ by $\hat{\nabla}$ and $\hat{R}$.
\begin{claim}\label{vector fields claim}
There exist vector fields $\{\hat{e}_{3},\cdots,\hat{e}_{2n}\}$ along $\hat{\gamma}_{\e,\ell}$ such that
\begin{itemize}\setlength{\itemsep}{1mm}
\item $\{\hat{\gamma}',J\hat{\gamma}',\hat{e}_{3},\cdots,\hat{e}_{2n}\}$ are orthonormal;
\item $ |\hat{\nabla}\hat{e}_{i}|^{2}\leq C\e^{2}\rho^{-4}e^{-2/\rho}$ for $i=3,\cdots,2n$.
\end{itemize}
\end{claim}

\begin{proof}[Proof of Claim \ref{vector fields claim}]
At the point $\hat{\gamma}_{\e,\ell}(-\hat{\ell })=\gamma(-\ell)$, we choose tangent vectors $\{v_{i}\}_{i=3}^{2n}$ such that $\{\hat{\gamma}_{\e,\ell}',J\hat{\gamma}_{\e,\ell}',v_{3},\cdots,v_{2n}\}$ are orthonormal basis of $T_{\hat{\gamma}_{\e,\ell}(-\hat{\ell})}M$. Let $\hat{v}_{i}$ be the $\hat g_\e$-parallel transportation of $v_{i}$ along $\hat{\gamma}_{\e,\ell}$. Define
\[
\hat{w}_{i} = \hat{v}_{i}-\hat{g}_\e(\hat{v}_{i},J\hat{\gamma}_{\e,\ell}')J\gamma_{\e,\ell}' \ \ \text{for $i=3,\cdots,2n$}.
\]
It is clear that
\[
\hat{g}_{\ve}(\hat{w}_{i},\hat{\gamma}_{\e,\ell}') = \hat{g}_{\ve}(\hat{w}_{i},J\hat{\gamma}_{\e,\ell}') = 0 \ \ \text{for $i=3,\cdots,2n$}.
\]
Using Lemma \ref{Conformal lemma} (a), $h=\ve\vp\leq\ve$ and $|\nabla h|^2\leq C\e^{2}\rho^{-4}e^{-2/\rho}$ (see \eqref{eqn:rho-estim}), we have $ |\hat{\nabla}J|^{2}\leq C\e^{2}\rho^{-4}e^{-2/\rho}$ and so
\[
|\hat{\nabla}\hat{w}_{i}|^{2} \leq C|\hat{\nabla}J|^{2} \leq C\e^{2}\rho^{-4}e^{-2/\rho} \ \ \text{for $i=3,\cdots,2n$}.
\]
Apply Gram-Schmidt orthonormalization to $\{\hat{w}_{3},\cdots,\hat{w}_{2n}\}$, we obtain the required $\{\hat{e}_{3},\cdots,\hat{e}_{2n}\}$.
\end{proof}

Let $\eta_{\hat{\ell}}:\mathbb{R}\to[0,+\infty)$ be an one-variable non-negative function such that
\begin{itemize}\setlength{\itemsep}{1mm}
\item $\eta_{\hat{\ell}}(t)=1$ when $t\in[-A_{0}-1,A_{0}+1]$;
\item $\eta_{\hat{\ell}}(t)=0$ when $t\in(-\infty,-\hat{\ell}]\cup[\hat{\ell},+\infty)$;
\item $|\eta_{\hat{\ell}}'(t)|\leq 2/\hat{\ell}$,
\end{itemize}
where $A_{0}$ is the uniform constant in Claim \ref{hat a claim}. For $2\leq i\leq 2n$, define
\[
\hat{V}_{i}(t) = \eta_{\hat{\ell}}(t)\hat{e}_{i}(t), \ \
\hat{\phi}_{i}(t,s) = \mathrm{exp}_{\hat{\gamma}_{\e,\ell}(t)}(s\hat{V}_{i}(t)), \ \
\hat{L}_{i}(s) = \mathrm{Length}_{\hat{g}_{\ve}}(\hat{\phi}_{i}(\cdot,s)).
\]
Since $\hat{\gamma}_{\e,\ell}|_{[-\hat{\ell},\hat{\ell}]}$ is minimizing, then $\hat{L}(s)$ achieves its minimum at $0$. The second variation formula of arc length gives:
\[
0 \leq \frac{d^{2}}{ds^{2}}\bigg|_{s=0}\hat{L}_{i}(s) =
\int_{-\hat{\ell}}^{\hat{\ell}}\left(|\hat{\nabla}\hat{V}_{i}|^{2}-\hat{R}(\hat{V}_{i},\hat{\gamma}_{\e,\ell}',\hat{\gamma}_{\e,\ell}',\hat{V}_{i})\right)dt.
\]
We compute
\[
|\hat{\nabla}\hat{V}_{i}|^{2}
= |\eta_{\hat{\ell}}'\hat{e}_{i}+\eta_{\hat{\ell}}\hat{\nabla}\hat{e}_{i}|^{2}
\leq 2|\eta_{\hat{\ell}}'|^{2}+2\eta_{\hat{\ell}}^{2}|\hat{\nabla}\hat{e}_{i}|^{2}
\leq \frac{8}{\hat{\ell}^{2}}+\eta_{\hat{\ell}}^{2}\cdot C\e^2 \rho^{-4}e^{-2/\rho}
\]
and so
\[
0 \leq \frac{d^{2}}{ds^{2}}\bigg|_{s=0}\hat{L}_{i}(s) \leq \frac{16}{\hat{\ell}}+\int_{-\hat{\ell}}^{\hat{\ell}}\eta_{\hat{\ell}}^{2}\left(C\e^2 e^{-2/\rho}\rho^{-4}-\hat{R}(\hat{e}_{i},\hat{\gamma}_{\e,\ell}',\hat{\gamma}_{\e,\ell}',\hat{e}_{i})\right)dt.
\]
Since $\alpha>0$ and $\alpha+\beta>0$, then
\begin{equation*}
\begin{split}
0 \leq {} & \frac{d^{2}}{ds^{2}}\bigg|_{s=0}\left((\alpha+\beta)\hat{L}_{2}(s)+\alpha\sum_{i=3}^{2n}\hat{L}_{i}(s)\right) \\
\leq {} & \Big((\alpha+\beta)+(2n-2)\alpha\Big)\cdot\frac{16}{\hat{\ell}}+\int_{-\hat{\ell}}^{\hat{\ell}}\eta_{\hat{l}}^{2}\cdot C\e^2 \rho^{-4}e^{-2/\rho}dt \\
& -\int_{-\hat{\ell}}^{\hat{\ell}}\eta_{\hat{l}}^{2}\Big((\alpha+\beta)\hat{R}(J\hat{\gamma}_{\e,\ell}',\hat{\gamma}_{\e,\ell}',\hat{\gamma}_{\e,\ell}',J\hat{\gamma}_{\e,\ell}')
+\alpha\sum_{i=3}^{2n}\hat{R}(\hat{e}_{i},\hat{\gamma}_{\e,\ell}',\hat{\gamma}_{\e,\ell}',\hat{e}_{i})\Big)\,dt \\
\leq {} & \frac{C}{\hat{\ell}}
+\int_{-\hat{\ell}}^{\hat{\ell}}\eta_{\hat{\ell}}^{2}\Big(C\e^2\rho^{-4} e^{-2/\rho}-\hat{\mathcal{C}}_{\alpha,\beta}(\hat{\gamma}_{\e,\ell}')\Big)\,dt.
\end{split}
\end{equation*}
Combining this with $e^{-1/\rho}\leq\rho$ and \eqref{upper bound of rho},
\[
\begin{split}
\int_{-\hat{\ell}}^{\hat{\ell}}\eta_{\hat{\ell}}^{2} \,\hat{\mathcal{C}}_{\a,\b}(\hat\gamma_{\e,\ell}')\,dt
\leq {} & \frac{C}{\hat{\ell}}+\int_{-\hat{\ell}}^{\hat{\ell}}\eta_{\hat{\ell}}^{2}\cdot C\e^2 \rho^{-4}e^{-2/\rho} dt \\
\leq {} & \frac{C}{\hat{\ell}}+\int_{-\hat{\ell}}^{\hat{\ell}}\eta_{\hat{\ell}}^{2}\cdot C\e^2 \rho^{-4}e^{-1/\rho}\cdot \rho\, dt \\
\leq {} & \frac{C}{\hat{\ell}}+\int_{-\hat{\ell}}^{\hat{\ell}}\eta_{\hat{\ell}}^{2}\cdot C\e^2 \rho^{-4}e^{-1/\rho}\cdot 4\mu_{0}r^{2}\, dt.
\end{split}
\]
Together with \eqref{lower bound of C a b},
\begin{equation*}
\int_{-\hat{\ell}}^{\hat{\ell}}\eta_{\hat{\ell}}^{2}\cdot\e r^{2}\Lambda_{0}^{-1}\rho^{-4}e^{-1/\rho}dt
\leq \frac{C}{\hat{\ell}}+\int_{-\hat{\ell}}^{\hat{\ell}}\eta_{\hat{\ell}}^{2}\cdot C\e^2 \rho^{-4}e^{-1/\rho}\cdot 4\mu_{0}r^{2}\, dt.
\end{equation*}
We might assume that $\ve\leq(8C\mu_{0}\Lambda_{0})^{-1}$ (which is uniform and independent of $\ell$), then
\begin{equation*}
\frac{1}{2}\int_{-\hat{\ell}}^{\hat{\ell}}\eta_{\hat{\ell}}^{2}\cdot\e r^{2}\Lambda_{0}^{-1}\rho^{-4}e^{-1/\rho}dt
\leq \frac{C}{\hat{\ell}}.
\end{equation*}
Combing this with $|\hat{\ell}-\ell|\leq C\ve$,
\begin{equation}\label{inequality C ell}
\e r^{2}\Lambda_{0}^{-1}\int_{-\hat{\ell}}^{\hat{\ell}}\eta_{\hat{\ell}}^{2}\rho^{-4}e^{-1/\rho}dt
\leq \frac{C}{\ell}.
\end{equation}
Thanks to Claim \ref{hat a claim}, there exists $\hat{a}_{\e,\ell}$ such that $d_{g}(z,\hat{\gamma}_{\e,\ell}(\hat{a}_{\e,\ell}))<(1+\mu_{0}-\delta_{0})r$ with $|\hat{a}_{\e,\ell}|\leq A_{0}$. When $t\in(\hat{a}_{\e,\ell}-\delta_{0} r/4,\hat{a}_{\e,\ell}+\delta_{0} r/4)$, using $g=e^{2\e\varphi}\hat{g}_{\ve}\leq e^{2\ve}\hat{g}_{\ve}\leq 4\hat{g}_{\ve}$ (by assuming $\ve<1/2$) we have
\[
\begin{split}
d_{g}(z,\hat{\gamma}_{\e,\ell}(t))
\leq {} & d_{g}(z,\hat{\gamma}_{\e,\ell}(\hat{a}_{\e,\ell}))+d_{g}(\hat{\gamma}_{\e,\ell}(\hat{a}_{\e,\ell}),\hat{\gamma}_{\e,\ell}(t)) \\
< {} & (1+\mu_{0}-\delta_{0})r+2d_{\hat{g}_{\ve}}(\hat{\gamma}_{\e,\ell}(\hat{a}_{\e,\ell}),\hat{\gamma}_{\e,\ell}(t)) \\
= {} & (1+\mu_{0}-\delta_{0})r+2\cdot|t-\hat{a}_{\e,\ell}| \\
< {} & (1+\mu_{0}-\delta_{0}/2)r
\end{split}
\]
and so
\[
\rho(\hat{\gamma}_{\e,\ell}(t)) \geq \rho_3 := (1+\mu_{0})^2r^2 -(1+\mu_{0}-\delta_{0}/2)^2r^2.
\]
By \eqref{upper bound of rho}, $r\leq r_{0}<1$ and $\mu_{0}\in(0,1/8)$, we have
\[
\rho(\cdot) < 1/2 \ \ \text{on $\hat{\gamma}_{\e,\ell}$}.
\]
It then follows that
\[
\rho_3 \leq \rho(\hat{\gamma}_{\e,\ell}(t)) < 1/2 \ \ \text{for $t\in(\hat{a}_{\e,\ell}-\delta_{0} r/4,\hat{a}_{\e,\ell}+\delta_{0} r/4)$}.
\]
Combining this with $s^{-4}e^{-1/s}$ is increasing for $s\in(0,1/2)$, and using \eqref{inequality C ell},
\[
\begin{split}
\e r^{2}\Lambda_{0}^{-1}\cdot \rho_3^{-4}e^{-1/\rho_3}\cdot\delta_{0} r
\leq {} & \e r^{2}\Lambda_{0}^{-1}\int_{\hat{a}_{\e,\ell}-\delta_{0} r/4}^{\hat{a}_{\e,\ell}+\delta_{0} r/4} \rho^{-4}e^{-1/\rho}dt \\
\leq {} & \e r^{2}\Lambda_{0}^{-1}\int_{-A_{0}-1}^{A_{0}+1}\rho^{-4}e^{-1/\rho}dt \\
\leq {} & \e r^{2}\Lambda_{0}^{-1}\int_{-\hat{\ell}}^{\hat{\ell}}\eta_{\hat{\ell}}^{2}\rho^{-4}e^{-1/\rho}dt \\
\leq {} & \frac{C}{\ell}.
\end{split}
\]
Hence, we obtain a contradiction if $\ell$ is sufficiently large (depending only on $\e$, $r$, $n$, $\alpha$ and $\beta$).

\bigskip
\noindent
{\bf Step 5.} The constructed $g$-geodesic line $\hat\gamma=\displaystyle{\lim_{\e\to0}\lim_{\ell\to+\infty}}\hat\gamma_{\e,\ell}$ satisfies the required properties (a)-(c) in Proposition \ref{prop:technical proposition}.
\bigskip

It is clear from construction that $\hat\gamma$ satisfies (a) from {\bf Step 4}. To prove (b) and (c), it suffices to show
\begin{enumerate}\setlength{\itemsep}{1mm}
\item[(b')] $\mathcal{B}_{+}(\hat{\gamma}_{\e,\ell}(t))+\mathcal{B}_{-}(\hat{\gamma}_{\e,\ell}(t))=O(\e)$;
\item[(c')] $\mathcal{B}_{+}(\hat{\gamma}_{\e,\ell}(t_{1}))-\mathcal{B}_{+}(\hat{\gamma}_{\e,\ell}(t_{2}))=t_{2}-t_{1}+O(\e)$ for all $t_1,t_2\in\mathbb{R}$.
\end{enumerate}
To prove (b'), we compute
\begin{equation}\label{B length hat gamma}
\begin{split}
\mathcal{B}_{+}(\hat{\gamma}_{\e,\ell}(t))
\leq {} & \lim_{s\to+\infty} \big( d_{g}(\hat{\gamma}_{\e,\ell}(t),\hat{\gamma}_{\e,\ell}(\hat{\ell}))+d_{g}(\hat{\gamma}_{\e,\ell}(\hat{\ell}),\gamma_{0}(s))-s \big) \\
= {} & d_{g}(\hat{\gamma}_{\e,\ell}(t),\hat{\gamma}_{\e,\ell}(\hat{\ell}))+\lim_{s\to+\infty}\big( d_{g}(\gamma(\ell),\gamma_{0}(s))-s \big) \\
\leq {} & \mathrm{Length}_{g}(\hat{\gamma}_{\e,\ell}|_{[t,\hat{\ell}]})+\mathcal{B}_{+}(\gamma(\ell)).
\end{split}
\end{equation}
Using $\hat{g}_\e=g$ outside $B_g(z,(1+\mu_{0})r)$ and \eqref{eqn:metric-compare}, we obtain
\begin{equation}\label{length hat gamma}
\begin{split}
\quad& \mathrm{Length}_{g}(\hat{\gamma}_{\e,\ell}|_{[t,\hat{\ell}]}) \\[0.5mm]
= {} & \mathrm{Length}_{\hat g_\e}(\hat{\gamma}_{\e,\ell}|_{[t,\hat{\ell}]})
+\int_{\hat{\gamma}_{\e,\ell}|_{[t,\hat{\ell}]}\cap B_g(z,(1+\mu_{0})r)}(|\hat{\gamma}_{\e,\ell}'|_{g}-|\hat{\gamma}_{\e,\ell}'|_{\hat{g}_\e})dt \\
\leq {} & \hat{\ell}-t+\left(e^{\ve(1+\mu_{0})^{2}r^{2}}-1\right)\cdot
\mathrm{Length}_{\hat g_\e}\big(\hat{\gamma}_{\e,\ell}|_{[t,\hat{\ell}]}\cap B_g(z,(1+\mu_{0})r)\big) \\[1.5mm]
\leq {} & \hat{\ell}-t+\left(e^{\ve(1+\mu_{0})^{2}r^{2}}-1\right)\cdot
\mathrm{Diam}_{\hat{g}_{\ve}}\big(B_g(z,(1+\mu_{0})r)\big) \\[1.5mm]
\leq {} & \hat{\ell}-t+\left(e^{\ve(1+\mu_{0})^{2}r^{2}}-1\right)\cdot
\mathrm{Diam}_{g}\big(B_g(z,(1+\mu_{0})r)\big) \\[1.5mm]
\leq {} & \hat{\ell}-t+C\e.
\end{split}
\end{equation}
Combining this with \eqref{B length hat gamma}, we obtain
\begin{equation}\label{B hat gamma upper bound 1}
\mathcal{B}_{+}(\hat{\gamma}_{\e,\ell}(t)) \leq \hat{\ell}-t+C\e+\mathcal{B}_{+}(\gamma(\ell)).
\end{equation}
By the similar argument of \eqref{B length hat gamma} and \eqref{length hat gamma}, we have
\[
\mathcal{B}_{-}(\hat{\gamma}_{\e,\ell}(t))
\leq \mathrm{Length}_{g}(\hat{\gamma}_{\e,\ell}|_{[-\hat{\ell},t]})+\mathcal{B}_{-}(\gamma(-\ell))
\]
and
\begin{equation}\label{length upper bound}
\mathrm{Length}_{g}(\hat{\gamma}_{\e,\ell}|_{[-\hat{\ell},t]}) \leq \hat{\ell}+t+C\e.
\end{equation}
Then
\begin{equation}\label{B hat gamma upper bound 2}
\mathcal{B}_{-}(\hat{\gamma}_{\e,\ell}(t)) \leq \hat{\ell}+t+C\e+\mathcal{B}_{-}(\gamma(-\ell)).
\end{equation}
Recall that the geodesic line $\gamma$ satisfies
\begin{enumerate}\setlength{\itemsep}{1mm}
\item[(i)] $\mathcal{B}_{+}+\mathcal{B}_{-}=0$ on $\gamma$;
\item[(ii)] $\mathcal{B}_{+}(\gamma(t_{1}))-\mathcal{B}_{+}(\gamma(t_{2}))=t_{2}-t_{1}$ for all $t_1,t_2\in\mathbb{R}$.
\end{enumerate}
Together with $\mathcal{B}_{+}+\mathcal{B}_{-}\geq0$, \eqref{B hat gamma upper bound 1} and \eqref{B hat gamma upper bound 2}, we obtain
\[
\begin{split}
0 \leq {} & \mathcal{B}_{+}(\hat{\gamma}_{\e,\ell}(t))+\mathcal{B}_{-}(\hat{\gamma}_{\e,\ell}(t)) \\
\leq {} & 2\hat{\ell}+C\e+\mathcal{B}_{+}(\gamma(\ell))+\mathcal{B}_{-}(\gamma(-\ell)) \\
= {} & 2\hat{\ell}+C\e+\mathcal{B}_{+}(\gamma(\ell))-\mathcal{B}_{+}(\gamma(-\ell)) \\
= {} & 2\hat{\ell}+C\e-2\ell \leq C\e,
\end{split}
\]
where we used $|\hat{\ell}-\ell|\leq C\e$ in the last line. Then we obtain (b').

\medskip

For (c'), we estimate the lower bound of $\mathcal{B}_{+}(\hat{\gamma}_{\e,\ell}(t))$:
\[
\begin{split}
\mathcal{B}_{+}(\hat{\gamma}_{\e,\ell}(t))
\geq {} & \lim_{s\to+\infty}\big( d_{g}(\hat{\gamma}_{\e,\ell}(-\hat{\ell}),\gamma_{0}(s))-d_{g}(\hat{\gamma}_{\e,\ell}(t),\hat{\gamma}_{\e,\ell}(-\hat{\ell}))-s \big) \\
= {} & \lim_{s\to+\infty}\big( d_{g}(\gamma(-\ell),\gamma_{0}(s))-s\big)-d_{g}(\hat{\gamma}_{\e,\ell}(t),\hat{\gamma}_{\e,\ell}(-\hat{\ell})) \\
\geq {} & \mathcal{B}_{+}(\gamma(-\ell))-\mathrm{Length}_{g}(\hat{\gamma}_{\e,\ell}|_{[-\hat{\ell},t]}).
\end{split}
\]
Using \eqref{length upper bound}, we see that
\[
\mathcal{B}_{+}(\hat{\gamma}_{\e,\ell}(t)) \geq \mathcal{B}_{+}(\gamma(-\ell))-\hat{\ell}-t-C\e.
\]
Combining this with \eqref{B hat gamma upper bound 1}, (ii) and $|\hat{\ell}-\ell|\leq C\e$,
\[
\begin{split}
& \mathcal{B}_{+}(\hat{\gamma}_{\e,\ell}(t_{1}))-\mathcal{B}_{+}(\hat{\gamma}_{\e,\ell}(t_{2})) \\
\geq {} & \mathcal{B}_{+}(\gamma(-\ell))-\hat{\ell}-t_{1}-C\e-\big(\hat{\ell}-t_{2}+C\e+\mathcal{B}_{+}(\gamma(\ell))\big) \\
= {} & \mathcal{B}_{+}(\gamma(-\ell))-\mathcal{B}_{+}(\gamma(\ell))-2\hat{\ell}+t_{2}-t_{1}-C\e \\
= {} & 2\ell-2\hat{\ell}+t_{2}-t_{1}-C\e \\
\geq {} & t_{2}-t_{1}-C\e.
\end{split}
\]
Interchanging $t_{1}$ and $t_{2}$, we obtain (c'). This completes the proof of Proposition \ref{prop:technical proposition}.
\end{proof}

\section{Vanishing results}\label{sec:vanishing}

In \cite{ChuLeeTam2022}, the first-, second-named authors and Tam considered the compact \K manifolds with $\mathcal{C}_{\a,\b}>0$, and showed that all Hodge numbers vanish by using Ni's co-mass technique \cite{Ni2021b}. In this section, we wish to prove some vanishing results for \K manifolds with either quasi-positive $\mathcal{C}_{\a,\b}$, or non-negative $\mathcal{C}_{\a,\b}$ and special holonomy group, which will be used in the proof of Theorem \ref{thm:structure} and \ref{thm:quasi-positive-mixed} (see Section \ref{sec:proof and applications}).

We first need an analogous computation to the proof of \cite[Proposition 7.1]{ChuLeeTam2022} for conformally \K metrics. We will use $\|\xi\|_{0,g}$ to denote the co-mass of a holomorphic $(p,0)$-form $\xi$ with respect to the Hermitian metric $g$, which is a function globally defined on $M$, see the discussion in \cite[Section 3]{Ni2021b}.

\begin{lma}\label{lma:contra-vanishing-twisted}
Let $(M,g)$ be a compact \K metric manifold with $\mathcal{C}_{\a,\b}(x)\geq \sigma(x)$ for some function $\sigma(x)$, $\xi$ be a non-zero holomorphic $(p,0)$-form and $\hat{g}=e^{\psi}g$ be a conformally K\"ahler metric. Suppose $\a>0$, $\a+\frac{2\b}{p+1}>0$ and $\|\xi\|_{0,\hat{g}}$ (the co-mass of $\xi$ with respect to $\hat g$) attains its maximum at $x_0\in M$. Then we have at $x_0$,
\[
\a \cdot \Delta_g \psi +\frac{2\b}{p+1}\sum_{i=1}^p \psi_{i\bar i}\geq \sigma(x_0)
\]
for some $g$-unitary frame at $x_0$. Here $\Delta_g f=\tr_g \ddb f$ denotes the complex Laplacian for smooth function $f$ in $M$.
\end{lma}

\begin{proof}
Since $\hat g=e^{\psi}g$, then
\[
\log\|\xi\|^2_{0,\hat g}=\log\|\xi\|^2_{0, g}-p\psi.
\]
Following the barrier argument in  \cite[Section 3]{Ni2021b}, we see that at $x_0$,
\begin{equation}\label{eqn:twist-R}
\begin{split}
0\geq \sum_{i=1}^p R_{v\bar vi\bar i}-p\psi_{v\bar v}
\end{split}
\end{equation}
for any $v\in T^{1,0}_{x_0}M$ and for some $g$-unitary frame $\{\frac{\partial}{\partial z^k}\}_{k=1}^n$ at $x_0$. Denote $\Sigma=\mathrm{span}\{\frac{\partial}{\partial z^i}: 1\leq i\leq p \}$. Using Berger's averaging trick and $\mathcal{C}_{\a,\b}(Z)\geq\sigma(x_{0})$,
\begin{equation}\label{eqn:average-trick}
\begin{split}
\sigma(x_0) \leq {} & \fint_{Z\in \Sigma,\,|Z|_g=1} \a \Ric(Z,\bar Z)+\b H(Z)\, d\theta(Z)\\
= {} & \frac{\a}{p}\sum_{i=1}^p R_{i\bar i}+\frac{2\b}{p(p+1)}\sum_{j=1}^p\sum_{i=1}^p R_{i\bar ij\bar j}\\
= {} & \frac{\a}{p}\sum_{j=p+1}^n\sum_{i=1}^p R_{i\bar ij\bar j}+\frac1p\left(\a +\frac{2\b}{p+1} \right)\sum_{j=1}^p\sum_{i=1}^p R_{i\bar ij\bar j}.
\end{split}
\end{equation}
Combining this with $\a>0$, $\a+\frac{2\b}{p+1}>0$ and \eqref{eqn:twist-R}, we obtain
\[
\sigma(x_0) \leq \a\sum_{j=p+1}^n \psi_{j\bar j} +\left( \a+\frac{2\b}{p+1}\right) \sum_{j=1}^p \psi_{j\bar j}
= \a \cdot \Delta \psi +\frac{2\b}{p+1}\sum_{j=1}^p \psi_{j\bar j},
\]
as required.
\end{proof}

Using a argument similar to the proof of Proposition \ref{prop:technical proposition}, we extend the vanishing theorem \cite[Theorem 7.1]{ChuLeeTam2022} to the quasi-positive case.

\begin{thm}\label{thm:vanishing}
Suppose $(M^n,g)$ is a compact K\"ahler manifold with quasi-positive $\mathcal{C}_{\a,\b}$ where $\a>0$ and $\a+\frac{2\b}{p+1}>0$. Then the Hodge number $h^{p,0}(M)=0$ for all $1\leq p\leq n$. In particular, if $\a+\b>0$, then $h^{p,0}(M)=0$ for all $1\leq p\leq n$.
\end{thm}

\begin{proof}
We argue by contradiction. Suppose there exists a non-zero holomorphic $(p,0)$-form $\xi$ on $M$. By the following two claims, we will show that the co-mass $\|\xi\|_{0,g}$ is constant on $M$, which implies $\|\xi\|_{0,g}$ attains its maximum at each point. Since $\mathcal{C}_{\a,\b}$ is quasi-positive, then at the point $x_{0}$ such that $\mathcal{C}_{\a,\b}(x_{0})>0$, applying Lemma \ref{lma:contra-vanishing-twisted} with $\psi=0$ yields a contradiction.

\medskip

We next show the constancy of $\|\xi\|_{0,g}$. Since $M$ is compact, by scaling we might assume
\begin{equation}\label{eqn:Rm inj}
\sup_M|\Rm(g)|\leq 1 \ \ \text{and} \ \ \mathrm{inj}(M,g)\geq 2.
\end{equation}
In particular, the function $d^2_g(\cdot,z)$ is smooth on $B_g(z,2)$ for all $z\in M$.

\begin{claim}\label{claim:induc-vanishing}
There exists $\mu_{0}(n,\a,\b,p)\in (0,1)$ such that the following holds. For any $y\in M$, $\ell>1$ and $r\in(0,1]$, if there is $\hat x\in A_g(y,r,\ell r)=B_g(y,\ell r)\setminus \overline{B_g(y,r)}$ satisfying
\[
\|\xi\|_{0,g}(\hat x) = \sup_M \|\xi\|_{0,g},
\]
then there is $\tilde x\in B_g(y,\ell r-\mu_{0} r)$ satisfying
\[
\|\xi\|_{0,g}(\tilde x) = \sup_M \|\xi\|_{0,g}.
\]
\end{claim}

\begin{proof}[Proof of Claim \ref{claim:induc-vanishing}]
Since $r\in(0,1]$, we assume without loss of generality that $r=1$ by considering $r^{-2}g$ which still satisfies \eqref{eqn:Rm inj}. The constant $\mu_{0}$ is a small constant to be determined later. Let $f$ be the function on $\mathbb{R}$ given by
\[
f(s) =
\begin{cases}
\ e^{-1/s}  & \mbox{for $s>0$};\\
\ 0  & \mbox{for $s\leq0$}.
\end{cases}
\]
Write $s_{0}=d_{g}(\hat{x},y)\in(1,\ell)$ and let $\gamma:[0,s_0]\to M$ be a minimizing geodesic of unit length joining from $\hat x=\gamma(0)$ to $y=\gamma(s_0)$. Denote $z=\gamma(1)$ and then $\hat x\in A_g(z,1-\mu_{0},1+\mu_{0})=B_g(z,1+\mu_{0})\setminus\overline{B_g(z,1-\mu_{0})}$. Define $\psi(x)=-f\circ \rho(x)$ where $\rho(x)=(1+\mu_{0})^2-d_g^2(x,z)$. We write $d_{g}=d_{g}(\cdot,z)$ and compute
\begin{equation*}
\begin{split}
\psi_{i\bar j}
= {} & -(1-2\rho)\rho^{-4}e^{-1/\rho}\rho_{i}\rho_{\bar j}-\rho^{-2}e^{-1/\rho}\rho_{i\bar j} \\[1mm]
= {} & -4d_{g}^{2}\cdot(1-2\rho)\rho^{-4}e^{-1/\rho}(d_{g})_{i}(d_{g})_{\bar j}-\rho^{-2}e^{-1/\rho}(d_{g}^{2})_{i\bar j}.
\end{split}
\end{equation*}
From the Hessian comparison theorem, we see that $|(d_{g}^{2})_{i\bar j}|\leq C_{n}$.
Hence for $x\in  A_g(z,1-\mu_{0},1+\mu_{0})$, we have
\[
\begin{split}
& \a \cdot \Delta \psi +\frac{2\b}{p+1}\sum_{j=1}^p \psi_{j\bar j} \\
= {} & -4\a d_{g}^{2}\cdot(1-2\rho)\rho^{-4}e^{-1/\rho}-\a \rho^{-2}e^{-1/\rho} \Delta d_{g}^{2} \\[3.5mm]
& -\frac{2\b}{p+1}\sum_{j=1}^p \Big( 4d_{g}^{2}\cdot(1-2\rho)\rho^{-4}e^{-1/\rho}|(d_{g})_{j}|^{2}+\rho^{-2}e^{-1/\rho}(d_{g}^{2})_{j\bar j}\Big)\\
\leq {} & -4d_{g}^{2}\cdot(1-2\rho)\rho^{-4}e^{-1/\rho}\left(\alpha\cdot|\de d_{g}|^{2}+\frac{2\beta}{p+1}\sum_{j=1}^{p}|(d_{g})_{j}|^{2}\right)
+C\rho^{-2}e^{-1/\rho}.
\end{split}
\]
If $\beta\geq0$, then $\alpha\cdot|\de d_{g}|^{2}+\frac{2\beta}{p+1}\sum_{j=1}^{p}|(d_{g})_{j}|^{2}\geq\alpha\cdot|\de d_{g}|^{2}$. If $\beta<0$, then
\[
\begin{split}
\alpha\cdot|\de d_{g}|^{2}+\frac{2\beta}{p+1}\sum_{j=1}^{p}|(d_{g})_{j}|^{2}
\geq {} & \left(\alpha+\frac{2\beta}{p+1}\right)|\de d_{g}|^{2}-\frac{2\beta}{p+1}\sum_{j=p+1}^{n}|(d_{g})_{j}|^{2} \\
\geq {} & \left(\alpha+\frac{2\beta}{p+1}\right)|\de d_{g}|^{2}.
\end{split}
\]
In both cases, we have $\alpha\cdot|\de d_{g}|^{2}+\frac{2\beta}{p+1}\sum_{j=1}^{p}|(d_{g})_{j}|^{2}\geq C^{-1}$. On the other hand, on $A_g(z,1-\mu_{0},1+\mu_{0})$, we also have
\begin{equation}\label{eqn:d g rho}
(1-\mu_{0})^{2} \leq d_{g}^{2} \leq (1+\mu_{0})^{2} \ \, \text{and} \ \,
0\leq\rho\leq 4\mu_{0}.
\end{equation}
We might assume $\mu_{0}\leq1/16$ and then
\[
\begin{split}
& \a \cdot \Delta \psi +\frac{2\b}{p+1}\sum_{j=1}^p \psi_{j\bar j} \\
\leq {} & -4\left(1-\frac{1}{16}\right)^{2}\left(1-\frac{8}{16}\right)\rho^{-4}e^{-1/\rho}\left(\alpha\cdot|\de d_{g}|^{2}+\frac{2\beta}{p+1}\sum_{j=1}^{p}|(d_{g})_{j}|^{2}\right)
+C\rho^{-2}e^{-1/\rho} \\[2mm]
\leq {} & -C^{-1}\rho^{-4}e^{-1/\rho}+C\rho^{-2}e^{-1/\rho} \\[5mm]
= {} & -C^{-1}\rho^{-4}e^{-1/\rho}(1-C\rho^{2}).
\end{split}
\]
Using \eqref{eqn:d g rho} and choose $\mu_{0}$ sufficiently small such that
\[
1-C\rho^{2} \geq 1-16C\mu_{0}^{2} > 0,
\]
we obtain
\begin{equation}\label{eqn:cont-mu}
\a \cdot \Delta \psi +\frac{2\b}{p+1}\sum_{j=1}^p \psi_{j\bar j} < 0 \ \
\text{in $A_g(z,1-\mu_{0},1+\mu_{0})$}.
\end{equation}

\medskip

For $\ve>0$, consider the conformally \K metric $\hat g_\e=e^{\e\psi}g$ and suppose for any $\e>0$, the co-mass $\|\xi\|_{0,\hat g_\e}$ attains its maximum at $x_\e\in M$. Since $\hat{g}_{\e}=g$ outside $B_{g}(z,1+\mu_{0})$ and $\hat{g}_{\e}<g$ at $\hat{x}$, then
\[
\begin{split}
\sup_{M\setminus B_g\left(z,1+\mu_{0}\right)}\|\xi\|_{0,\hat g_\e}
= \sup_{M\setminus B_g\left(z,1+\mu_{0}\right)}\|\xi\|_{0,g}
\leq \|\xi\|_{0,g}(\hat x)
< \|\xi\|_{0,\hat g_\e}(\hat x)
\leq \|\xi\|_{0,\hat g_\e}(x_{\e}).
\end{split}
\]
This implies $x_\e\in B_g(z,1+\mu_{0})$. On the other hand, Lemma~\ref{lma:contra-vanishing-twisted} and \eqref{eqn:cont-mu} imply that  $x_\e\notin  A_g(z,1-\mu_{0},1+\mu_{0})$. Then we must have $x_\e\in \overline{B_g(z,1-\mu_{0})}$. Passing to a subsequence of $\ve\to0$, we assume $x_{\ve}\to\tilde{x}$ for some $\tilde x \in \overline{B_g(z,1-\mu_{0})}$. Together with $\hat{g}_{\ve}\to g$, we obtain $\|\xi\|_{0,g}(\tilde x)=\sup_M \|\xi\|_{0,g}$. Moreover, using the triangle inequality and $s_{0}\in(1,\ell)$,
\begin{equation*}
d_g(\tilde x,y)\leq d_g(\tilde  x,z)+d_g(z,y)\leq 1-\mu_{0}+s_0-1<\ell-\mu_{0}.
\end{equation*}
This completes the proof.
\end{proof}

We now use it to propagate the maximum of co-mass $\|\xi\|_0$.
\begin{claim}\label{claim:constant-comass}
The function  $\|\xi\|_{0,g}(x)$ is constant on $M$.
\end{claim}
\begin{proof}[Proof of Claim \ref{claim:constant-comass}]
It suffices to show that for any $x\in M$ and $r\in(0,1]$, we can find $\hat x\in \overline{B_g(x,r)}$ such that $\|\xi\|_{0,g}(\hat x)=\sup_M \|\xi\|_{0,g}$. Fix $x_1\in M$ so that $\|\xi\|_{0,g}(x_1)=\sup_M \|\xi\|_{0,g}$. Suppose $x_{1}\in B_{g}(x,\ell r)$ for some $\ell>1$. We might assume $x_1\in A_g(x,r,\ell r)$ (otherwise we are done). By Claim \ref{claim:induc-vanishing}, there exists $x_2 \in B_g(x,\ell r-\mu_{0} r)$ such that $\|\xi\|_{0,g}(x_2)=\sup_M \|\xi\|_{0,g}$. We might assume $\ell-\mu_{0}>1$ (otherwise we are done), then $x_2$ is either in $ A_g(x,r,\ell r-\mu_{0}r)$ or $\overline{B_g(x,r)}$. We are done in the latter case. Since $M$ is compact, by repeating the above procedure finitely, we are done.
\end{proof}
\end{proof}

We next consider compact \K manifolds with $\mathcal{C}_{\a,\b}\geq 0$ and $\mathrm{Hol}(M,g)=U(n)$. In this case, we combine ideas in Ni-Zheng \cite{NiZheng2018} with Campana-Demailly-Peternell \cite{CampanaDemaillyPeternell2015} to show that $h^{2,0}(M)=0$ and thus $M$ will be projective by the Kodaira Theorem \cite{Kodaira1954}.

\begin{thm}\label{thm:holo-vanishing}
Suppose $(M^{n},g)$ is a compact \K manifold with $\mathcal{C}_{\a,\b}\geq 0$ where $\alpha>0$, $\alpha+\beta>0$ and $\mathrm{Hol}(M,g)=U(n)$. Then $h^{2,0}(M)=0$. In particular, $M$ is projective.
\end{thm}

\begin{proof}
We will adapt some argument in \cite[Theorem 1.7]{NiZheng2018}. Suppose on the contrary and let $\xi$ be a non-zero global holomorphic $(2,0)$-form on $M$. Let $k$ be the largest integer so that $s=\xi^k$ is not identically zero on $M$. Then at any point $x\in M$ such that $s(x)\neq0$, there is a $g$-unitary frame $\{e_{i}\}_{i=1}^{n}$ (denote its dual coframe by $\{e^{i}\}_{i=1}^{n}$) such that at $x$,
\[
\xi = \lambda_{1}e^{1}\wedge e^{2}+\cdots+\lambda_{k}e^{2k-1}\wedge e^{2k}, \ \
s = k!(\lambda_{1}\cdots\lambda_{k})e^{1}\wedge\cdots\wedge e^{2k}
\]
for some $\lambda_{i}\neq0$. By the Bochner formula \cite[(4.1)]{NiZheng2018}, we have at $x$,
\begin{equation}\label{eqn:Bochner}
\left\langle \ddb |s|_{g}^{2},\frac{1}{\sqrt{-1}}v\wedge\bar{v}\right\rangle
= |\nabla_{v}s|_{g}^{2}+|s|_{g}^{2}\sum_{i=1}^{2k}R_{i\bar iv\bar v}
\end{equation}
for any $v\in T_{x}^{1,0}M$.

\begin{claim}\label{claim:constant-norm}
The function $|s|^2$ is constant on $M$.
\end{claim}
\begin{proof}[Proof of Claim \ref{claim:constant-norm}]
This is almost identical to that of Claim \ref{claim:constant-comass}. We only sketch the proof. Consider the conformally \K metric $\hat g=e^{\psi}g$. At the maximum point $x_{0}$ of $\log|s|^2_{\hat g}=\log|s|^2_{g}-2k\psi$, by \eqref{eqn:Bochner}, we have
$$\sum_{i=1}^{2k}R_{i\bar iv\bar v}-2k\psi_{v\bar v}\leq 0$$
for any $v\in T_{x_{0}}^{1,0}M$, which is an exact analogy of \eqref{eqn:twist-R}. Then we can repeat the argument of Lemma \ref{lma:contra-vanishing-twisted}, Claim \ref{claim:induc-vanishing} and Claim \ref{claim:constant-comass} to obtain Claim \ref{claim:constant-norm}.
\end{proof}

Since $|s|^2$ is constant and $\mathcal{C}_{\a,\b}\geq 0$, it follows from Berger's averaging trick (see \eqref{eqn:average-trick}) and \eqref{eqn:Bochner} that
\[
\begin{split}
0 \leq {} & \frac{\a}{2k}\sum_{j=2k+1}^n\sum_{i=1}^{2k} R_{i\bar ij\bar j}
+\frac{1}{2k}\left(\a +\frac{2\b}{2k+1} \right)\sum_{j=1}^{2k}\sum_{i=1}^{2k} R_{i\bar ij\bar j} \\
= {} & \frac{\a}{2k|s|^{2}}\sum_{j=2k+1}^n\big|\nabla_{\frac{\de}{\de z^{j}}}s\big|^{2}
+\frac{1}{2k|s|^{2}}\left(\a +\frac{2\b}{2k+1} \right)\sum_{j=1}^{2k}\big|\nabla_{\frac{\de}{\de z^{j}}}s\big|^{2}.
\end{split}
\]
Combining this with $\alpha>0$ and $\alpha+\beta>0$, we obtain $\nabla s=0$ and so $s$ is parallel on $M$. However, together with $\mathrm{Hol}(M,g)=U(n)$, $s$ must be zero, which is impossible.
\end{proof}

\section{Proof of main results and applications}\label{sec:proof and applications}

In this section, we apply results in earlier sections to prove the structure theorem (Theorem \ref{thm:structure}) and vanishing theorem (Theorem \ref{thm:quasi-positive-mixed}), and give some applications on K\"ahler manifolds with various known curvature conditions.

\subsection{Proof of main results}
We first prove the structure theorem for \K manifolds with $\mathcal{C}_{\a,\b}\geq 0$.

\begin{proof}[Proof of Theorem~\ref{thm:structure}]
Let $(\tilde M,\tilde g)$ be the universal cover of $(M,g)$. By applying Theorem \ref{thm:splitting} inductively, we see that $\tilde M^n$ splits isometrically as $N^{n-k}\times \mathbb{C}^k$, where $ N^{n-k}$ does not contain a geodesic line. Since $M$ is compact, it is standard that $N$ is also compact since otherwise we might use the deck transformation to construct a line in $N$, for example see the discussion of the structure theorem under $\Ric\geq 0$ in \cite[Section 7.3.3]{Petersen2016}.

By the de Rham decomposition \cite{deRham1952}, we write
\begin{equation*}
(N^{n-k},h) = \prod_{i=1}^N (X_i^{n_i},h_i)
\end{equation*}
induced by a decomposition of the holonomy representation in irreducible representations. By Berger's classification of holonomy groups \cite{Berger1955}, $X_i$ is either Calabi-Yau (i.e. $c_{1}(X_{i})=0$), Hermitian symmetric space of compact type or has holonomy group $U(n_i)$.

\begin{claim}\label{claim:pseudo}
Each $X_i$ admits a \K metric with $\mathrm{scal}>0$. In particular, the canonical line bundle $K_{X_i}$ is not pseudo-effective.
\end{claim}
\begin{proof}[Proof of Claim \ref{claim:pseudo}]
We omit the index $i$ for notational convenience. Using $\mathcal{C}_{\a,\b,h}\geq 0$ and Berger's averaging trick (see \eqref{eqn:average-trick}), we have
\[
\begin{split}
0 \leq {} & \fint_{|Z|=1} \a \Ric(Z,\bar Z)+\b H(Z)\, d\theta(Z)\\
= {} & \frac{\a}{n}\sum_{i=1}^n R_{i\bar i}+\frac{2\b}{n(n+1)}\sum_{j=1}^n\sum_{i=1}^n R_{i\bar ij\bar j} \\
= {} & \left(\frac{\a}{n}+\frac{2\b}{n(n+1)}\right)\mathrm{scal}_{h}.
\end{split}
\]
Combining this with $\alpha>0$ and $\alpha+\beta>0$, we obtain $\mathrm{scal}_{h}\geq 0$. We let $h(t)$ be a short-time solution to the \KR flow starting from $h$ on $X\times [0,T)$. By the strong maximum principle, either $\Ric(h(t))\equiv 0$ or $\mathrm{scal}_{h(t)}>0$ for all $t\in (0,T]$. In the former case, then the Ricci curvature of initial metric $h$ also vanishes. Combining this with $\mathcal{C}_{\a,\b,h}\geq 0$ and $\beta\neq0$, we obtain the holomorphic sectional curvature of $h$ vanishes and so $(X,h)$ must be flat. Since $X$ is simply connected and compact, this is impossible. Therefore, we must have $\mathrm{scal}_{h(t)}>0$ on $X$, and so $h(t)$ is the required K\"ahler metric.
\end{proof}

By  Claim~\ref{claim:pseudo}, $X_i$ cannot be Calabi-Yau manifold. On the other hand, it is well-known that compact Hermitian symmetric spaces are Fano. In particular, their second Hodge number vanishes. When $\mathrm{Hol}(X_i^{n_i},h_i)=U(n_i)$. It follows from Theorem~\ref{thm:holo-vanishing} that $h^{2,0}(X_i)=0$. This completes the proof by combining all cases.
\end{proof}

We now move on to consider the quasi-positive case and prove the vanishing theorem, which is an immediate consequence of Theorem \ref{thm:splitting} and \ref{thm:vanishing}.

\begin{proof}[Proof of Theorem \ref{thm:quasi-positive-mixed}]
Since $\mathcal{C}_{\a,\b}$ is positive at some point, the universal cover $(\tilde{M},\tilde{g})$ cannot split. Combining this with Theorem \ref{thm:splitting}, $(\tilde{M},\tilde{g})$ is also compact. Moreover, Theorem \ref{thm:vanishing} shows that the Hodge number $h^{p,0}(M)=h^{p,0}(\tilde{M})=0$ for all $1\leq p\leq n$. By the Kodaira Theorem \cite{Kodaira1954}, both $\tilde M$ and $M$ are projective. Using \cite[Lemma 1]{Kobayashi1961}, we can conclude that $M$ is simply connected by comparing their Euler characteristic numbers.
\end{proof}

\subsection{Applications}\label{subsec:applications}
The structure theorem and vanishing theorem obtained in this paper can be applied to \K manifolds obtained by Ni \cite{Ni2021a,Ni2021b}. Recall in \cite{Ni2021b} that on a \K manifold $(M,g)$, the curvature notion $\Ric^{+}$ defined by
\[
\Ric^+(X,\bar X) = \mathcal{C}_{1,1}(X) = \Ric(X,\bar X)+R(X,\bar X,X,\bar X)
\]
for unitary $X\in T^{1,0}M$. In a similar but different spirit, the concept of $k$-Ricci curvature $\Ric_k$ was also introduced by Ni \cite{Ni2021a} which serves as an interpolation between the Ricci curvature and holomorphic sectional curvature. More precisely, for any $p\in M$ and any $k$-dimensional subspace $U$ of $T_p^{1,0}M$, $\Ric_{k,U}$ is defined to be the Ricci tensor for $\mathrm{Rm}|_U$ (the restriction of curvature tensor $\mathrm{Rm}$ on $U$). We say that $\Ric_k$ is quasi-positive if $\Ric_{k,U}\geq 0$ for all $p\in M$ and $k$-dimensional subspace $U$ in $T_p^{1,0}M$, and is positive at some $p_0\in M$ for all $k$-dimensional subspace $U$ in $T_{p_0}^{1,0}M$. We refer readers to \cite[Section 2]{ChuLeeTam2022} for a detailed exposition.

In \cite{Ni2021b}, Ni showed that the \K manifold $M^n$ is simply connected and projective with $h^{p,0}=0$ for all $1\leq p\leq n$ if $\Ric^+>0$ or $\Ric_k>0$ for some $1\leq k\leq n$ (see also the work of Tang \cite{Tang2024}). From \cite[Lemma 2.2]{ChuLeeTam2022}, we see that Theorem \ref{thm:structure} and \ref{thm:quasi-positive-mixed} can be applied to compact K\"ahler manifolds with non-negative or quasi-positive $\Ric^+$ or $\Ric_k$ for $2\leq k\leq n-1$.

\appendix

\section{Conformal calculation}\label{sec:conformal calculation}

\begin{lma}\label{Conformal lemma}
Let $(M^{n},g,J)$ be a K\"ahler manifold and $\hat{g}=e^{-2h}g$ for some $f\in C^{\infty}(M)$. Then the following holds.
\begin{enumerate}\setlength{\itemsep}{1mm}
\item[(a)] $|\hat{\nabla}J|_{\hat{g}}^{2} \leq C_{n}e^{2h}|\nabla h|_{g}^{2}$ for some dimensional constant $C_{n}$.
\item[(b)] For $g$-unit vector $v$, define $\hat{g}$-unit vector $\hat{v}=e^{h}v$. Then
\[
\begin{split}
\mathcal{C}_{\alpha,\beta,\hat{g}}(\hat{v})
= {} e^{2h}\Big( & \mathcal{C}_{\alpha,\beta,g}(v)+\alpha\Delta h+(2n\alpha-2\alpha+\beta)\nabla^{2}h(v,v)+\beta\nabla^{2}h(Jv,Jv) \\
& +(2n\alpha-2\alpha+\beta)v(h)^{2}+\beta Jv(h)^{2}-(2n\alpha-2\alpha+\beta)|\nabla h|^{2} \Big).
\end{split}
\]
\end{enumerate}
Here $\hat{\nabla}$ denotes the Levi-Civita connection of $\hat{g}$ and $\mathcal{C}_{\a,\b,\hat{g}}$ denotes the mixed curvature of $\hat{g}$ defined in Definition \ref{def:c a b Hermitian}.
\end{lma}

\begin{proof}
Denote the curvature and Ricci tensor of $\hat{g}$ by $\hat{R}$ and $\widehat{\Ric}$. For any $g$-unit vector $w$, define $\hat{g}$-unit vector $\hat{w}=e^{h}w$. Direct calculation shows
\[
\hat{\nabla}_{v}w-\nabla_{v}w = -v(h)w-w(h)v+g(v,w)\nabla h
\]
and so
\[
\begin{split}
\hat{\nabla}_{\hat{v}}\hat{w}
= {} & \hat{\nabla}_{e^{h}v}(e^{h}w) \\
= {} & e^{2h}\big(v(h)w+\hat{\nabla}_{v}w\big) \\
= {} & e^{2h}\big(\nabla_{v}w-w(h)v+g(v,w)\nabla h\big).
\end{split}
\]
Replacing $\hat{w}$ by $J\hat{w}$,
\[
\hat{\nabla}_{\hat{v}}(J\hat{w})
= e^{2h}\big(\nabla_{v}(Jw)-Jw(h)v+g(v,Jw)\nabla h\big).
\]
We compute
\[
\begin{split}
& (\hat{\nabla}_{\hat{v}}J)\hat{w} \\
= {} & \hat{\nabla}_{\hat{v}}(J\hat{w})-J(\hat{\nabla}_{\hat{v}}\hat{w}) \\
= {} & e^{2h}\big(\nabla_{v}(Jw)-Jw(h)v+g(v,Jw)\nabla h\big) -e^{2h}\big(\nabla_{v}w-w(h)v+g(v,w)\nabla h\big) \\
= {} & e^{2h}\big((\nabla_{v}J)w-Jw(h)v+g(v,Jw)\nabla h+w(h)v-g(v,w)\nabla h\big).
\end{split}
\]
Combining this with $\nabla J=0$ (follows from the K\"ahlerity of $g$), we obtain (a).

\medskip

For (b), when $g(v,w)=0$, it is well-known that
\[
\hat{R}(\hat{v},\hat{w},\hat{w},\hat{v})
= e^{2h}\big( R(v,w,w,v)+\nabla^{2}h(v,v)+\nabla^{2}h(w,w)+v(h)^{2}+w(h)^{2}-|\nabla h|^{2} \big).
\]
It then follows that
\[
\widehat{\Ric}(\hat{v},\hat{v})
= e^{2h}\big( \Ric(v,v)+\Delta h+(2n-2)\nabla^{2}h(v,v)-(2n-2)|\nabla h|^{2}+(2n-2)v(h)^{2} \big)
\]
and
\[
\begin{split}
& \hat{R}(\hat{v},J\hat{v},J\hat{v},\hat{v}) \\
= {} & e^{2h}\big( R(v,Jv,Jv,v)+\nabla^{2}h(v,v)+\nabla^{2}h(Jv,Jv)+v(h)^{2}+Jv(h)^{2}-|\nabla h|^{2} \big).
\end{split}
\]
Combining the above, we obtain (b).
\end{proof}


\begin{thebibliography}{100}

\bibitem{Berger1955} Berger, M., {\sl Sur les groupes d'holonomie homog\`{e}ne des vari\'{e}t\'{e}s \`{a} connexion affine et des vari\'{e}t\'{e}s riemanniennes}, Bull. Soc. Math. France 83 (1955), 279--330.

\bibitem{Campana1992} Campana, F., {\sl Connexit\'e rationnelle des vari\'et\'es de Fano}, Ann. Sci. \'Ecole Norm. Sup. (4) 25 (1992), no. 5, 539--545.

\bibitem{CampanaDemaillyPeternell2015} Campana, F.; Demailly J.-P. ; Peternell, T., {\sl Rationally connected manifolds and semipositivity of the Ricci curvature}, Recent advances in algebraic geometry, 71--91. London Math. Soc. Lecture Note Ser., 417 Cambridge University Press, Cambridge, 2015.

\bibitem{CheegerGromoll1971} Cheeger, J.; Gromoll, D., {\sl The splitting theorem for manifolds of nonnegative Ricci curvature}, J. Differential Geometry 6 (1971), 119--128.

\bibitem{ChodoshEichmairMoraru2019} Chodosh, O.; Eichmair, M.; Moraru, V., {\sl A splitting theorem for scalar curvature}, Comm. Pure Appl. Math. 72 (2019), no. 6, 1231--1242.

\bibitem{ChuLeeTam2022} Chu, J.; Lee, M.-C.; Tam, L.-F., {\sl K\"ahler manifolds and mixed curvature}, Trans. Amer. Math. Soc. 375 (2022), no. 11, 7925--7944.

\bibitem{deRham1952} de Rham, G., {\sl Sur la reductibilit\'{e} d'un espace de Riemann}, Comment. Math. Helv. 26 (1952), 328--344.

\bibitem{Frankel1961} Frankel, T., {\sl Manifolds with positive curvature}, Pacific J. Math. 11 (1961), 165--174.

\bibitem{Hartshorne1970} Hartshorne, R., {\sl Ample subvarieties of algebraic varieties}, Notes written in collaboration with C. Musili. Lecture Notes in Math., Vol. 156. Springer-Verlag, Berlin-New York, 1970. xiv+256 pp.

\bibitem{HeierWong2020} Heier, G.; Wong, B., {\sl On projective K\"ahler manifolds of partially positive curvature and rational connectedness}, Doc. Math. 25 (2020), 219--238.

\bibitem{HowardSmythWu1981} Howard, A.; Smyth, B.; Wu, H., {\sl On compact K\"ahler manifolds of nonnegative bisectional curvature. I}, Acta Math. 147 (1981), no. 1--2, 51--56.

\bibitem{Kobayashi1961} Kobayashi, S., {\sl On compact K\"ahler manifolds with positive definite Ricci tensor}, Ann. of Math. (2) 74 (1961), 570--574.

\bibitem{Kodaira1954} Kodaira, K., {\sl On K\"ahler varieties of restricted type (an intrinsic characterization of algebraic varieties)}, Ann. of Math. (2) 60 (1954), 28--48.

\bibitem{KollarMiyaokaMori1992} Koll\'ar, J.; Miyaoka, Y.; Mori, S., {\sl Rational connectedness and boundedness of Fano manifolds}, J. Differential Geom. 36 (1992), no. 3, 765--779.

\bibitem{LeeLi2022} Lee, M.-C.; Li, K.-F., {\sl Deformation of Hermitian metrics}, Math. Res. Lett. 29 (2022), no. 5, 1485--1497.

\bibitem{Liu2013} Liu, G., {\sl $3$-manifolds with nonnegative Ricci curvature}, Invent. Math. 193 (2013), no. 2, 367--375.

\bibitem{Matsumura2020} Matsumura, S., {\sl On the image of MRC fibrations of projective manifolds with semi-positive holomorphic sectional curvature}, Pure Appl. Math. Q. 16 (2020), no. 5, 1419--1439.

\bibitem{Matsumura2022} Matsumura, S., {\sl On projective manifolds with semi-positive holomorphic sectional curvature}, Amer. J. Math. 144 (2022), no. 3, 747--777.

\bibitem{Mok1988} Mok, N., {\sl The uniformization theorem for compact K\"ahler manifolds of nonnegative holomorphic bisectional curvature}, J. Differential Geom. 27 (1988), no. 2, 179--214.

\bibitem{Mori1979} Mori, S., {\sl Projective manifolds with ample tangent bundles}, Ann. of Math. (2) 110 (1979), no. 3, 593--606.

\bibitem{Ni2021a} Ni, L., {\sl Liouville theorems and a Schwarz Lemma for holomorphic mappings between K\"ahler manifolds}, Comm. Pure Appl. Math. 74 (2021), no. 5, 1100--1126.

\bibitem{Ni2021b} Ni, L., {\sl The fundamental group, rational connectedness and the positivity of K\"ahler manifolds}, J. Reine Angew. Math. 774 (2021), 267--299.

\bibitem{NiZheng2018} Ni, L.; Zheng, F., {\sl Comparison and vanishing theorems for K\"ahler manifolds}, Calc. Var. Partial Differential Equations 57 (2018), no. 6, Paper No. 151, 31 pp.

\bibitem{Petersen2016} Petersen, P., {\sl Riemannian geometry}, Third edition. Grad. Texts in Math., 171. Springer, Cham, 2016. xviii+499 pp.

\bibitem{SiuYau1980} Siu, Y.-T.; Yau, S.-T., {\sl Compact K\"ahler manifolds of positive bisectional curvature}, Invent. Math. 59 (1980), no. 2, 189--204.

\bibitem{Tang2024} Tang, K., {\sl Quasi-positive curvature and vanishing theorems}, preprint, arXiv:2405.03895.

\bibitem{Yang2018} Yang, X., {\sl RC-positivity, rational connectedness and Yau's conjecture}, Camb. J. Math. 6 (2018), no. 2, 183--212.

\bibitem{Yang2020} Yang, X., {\sl Compact K\"ahler manifolds with quasi-positive second Chern-Ricci curvature}, preprint, arXiv:2006.13884.

\bibitem{Yau1982} Yau, S.-T., {\sl Problem section}. Seminar on Differential Geometry, pp. 669--706. Ann. of Math. Stud., No. 102. Princeton University Press, Princeton, NJ, 1982.

\bibitem{ZhangZhang2023} Zhang, S.; Zhang, X., {\sl On the structure of compact K\"ahler manifolds with nonnegative holomorphic sectional curvature}, preprint, arXiv:2311.18779.

\end{thebibliography}
\end{document}